\documentclass[11pt]{article}
\usepackage{amscd}
\usepackage{amsmath}
\usepackage{latexsym}
\usepackage{amsfonts}
\usepackage{amssymb}
\usepackage{amsthm}
\usepackage{graphicx}
\usepackage{verbatim}
\usepackage{mathrsfs}
\usepackage{enumerate}
\usepackage{hyperref}
\usepackage{fullpage}
\usepackage{color}

\theoremstyle{plain}
\theoremstyle{definition}\newtheorem{theorem}{Theorem}[section]
\theoremstyle{plain}\newtheorem{lemma}[theorem]{Lemma}
\theoremstyle{plain}
\theoremstyle{plain}\newtheorem{proposition}[theorem]{Proposition}
\theoremstyle{remark}\newtheorem{remark}{Remark}[section]

\newcommand{\norm}[1]{\left\|#1\right\|}

\begin{document}
\title{On the global well-posedness of a class of Boussinesq- Navier-Stokes systems}
\author{Changxing Miao$^1$ and  Liutang Xue$^2$\\
        \\
        \small{$^{1}$ Institute of Applied Physics and Computational Mathematics,}\\
        \small{P.O. Box 8009, Beijing 100088, P.R. China.}\\
        \small{(miao\_{}changxing@iapcm.ac.cn)}\\
        \small{$^2$  The Graduate School of China Academy of Engineering Physics,}\\
        {\small P.O. Box 2101, Beijing 100088, P.R. China.}\\
        {\small{(xue\_{}lt@163.com)}}}

\date{}
\maketitle

\date{}
\maketitle
\begin{abstract}
 In this paper we consider the following 2D Boussinesq-Navier-Stokes systems
\begin{equation*}
 \begin{split}
 \partial_{t}u+u\cdot\nabla u+\nabla
 p+  |D|^{\alpha}u &= \theta e_{2}   \\
 \partial_{t}\theta+u\cdot\nabla \theta+   |D|^{\beta}\theta &=0 \quad
 \end{split}
\end{equation*}
with $\textrm{div} u=0$ and $0<\beta<\alpha<1$. When $\frac{6-\sqrt{6}}{4}<\alpha< 1$, $1-\alpha<\beta\leq f(\alpha) $,
where $f(\alpha)$ is an explicit function as a technical bound,
we prove global well-posedness results for rough initial data. 
\end{abstract}



\section{Introduction}
\setcounter{section}{1}\setcounter{equation}{0}

The 2D generalized Boussinesq systems are of the forms
\begin{equation}\label{eq 1.1}
 \begin{cases}\partial_{t}u+u\cdot\nabla u+\nabla
 p+ \nu |D|^{\alpha}u=\theta e_{2}, \quad (t,x)\in \mathbb{R}^{+}\times \mathbb{R}^{2} \\
 \partial_{t}\theta+u\cdot\nabla \theta+ \kappa |D|^{\beta}\theta=0 \\
 \textrm{div} u=0 \\
 u|_{t=0}=u^{0},\quad \theta|_{t=0}=\theta^{0}, \end{cases}
\end{equation}
where $\nu\geq 0,\kappa\geq 0$, $(\alpha,\beta)\in[0,2]^{2}$ and $|D|^{\alpha}$ is defined via the
Fourier transform
$$\mathcal{F}(|D|^{\alpha}f)(\xi)=|\xi|^{\alpha}\mathcal{F}(f)(\xi).$$
These systems are simple models widely used in the modeling of the oceanic and atmospheric motions (see e.g.\cite{ref Pedlosky}).
Here, the divergence-free vector field $u=(u^{1},u^{2})$ denotes the velocity, scalar functions $\theta$, $p$ denote the
temperature and the pressure respectively, the absolute constants $\nu,\kappa$ can be seen as the inverse of Reynolds numbers.
The term $\theta e_{2}$ in the velocity equation, with $e_{2}$ the canonical vector $(0,1)$, models
the effect of gravity on the fluid motion. If $\theta^{0}=0$, the systems are reduced to the 2D generalized Navier-Stokes(Euler) equations.
Clearly, due to the maximum principle for the vorticity and the B-K-M criterion in \cite{Beale},
smooth solutions of these two-dimensional systems are global in time.

From a mathematical view, the fully viscous model with $\nu>0,\kappa>0, \alpha=\beta=2$ is the simplest one to study. It acts very similar to the
2D Navier-Stokes equation and
similar global results can be achieved. On the other hand, the most difficult one for the mathematical study is the inviscid model, that is when $\nu=\kappa=0$.
Up to now, only local existence
results can be proven.

Here we focus on the cases where the dissipation effect in the velocity equation plays a dominant role.
The most typical models are those with the diffusion effect in the temperature equation neglected ($\kappa=0, \nu>0$),
and there have been some recent important works on these Boussinesq systems. For the case with the full viscosity, i.e. when $\alpha=2$,
global well-posedness results can be established in various functional spaces.
In \cite{ref chae}, Chae proved that for large initial data $(u^{0},\theta^{0})\in H^{s}\times H^{s}$ with $s>2$ the system is global well-posed. See also \cite{ref HouLi}.
Later on, Hmidi-Keraani
in \cite{ref Hmidi-Ker1} showed global well-posedness for less regular data $(u^{0},\theta^{0})\in H^{s}\times H^{s}$ with $s>0$.
In \cite{ref Dan-Paicu2}, Danchin-Paicu proved the unconditional
uniqueness in the energy space $L^{2}\times L^{2}$. For the case with weaker dissipation, i.e. when $1\leq\alpha<2$, the problem is also solvable.
When $\alpha\in]1,2[$, as in \cite{ref Hmidi-Ker1}
through taking advantage of the maximal regularity estimates for the semi-group $e^{-t|D|^{\alpha}}$,
one can prove the global well-posedness.
For the subtle critical case $\alpha=1$, Hmidi-Keraani-Rousset in \cite{ref HmidiKR-BNS} proved the global result for the rough data
through exploiting the new structural properties of the system solved by vorticity $\omega$ and temperature $\theta$.

One natural problem is how to establish the global well-posedness for the problem \eqref{eq 1.1} when
 we further weaken the dissipation effect in the velocity equation to the $\alpha<1$ case.
 It seems that introducing the diffusion effect in the temperature equation ($\kappa>0$) is necessary and meanwhile $\beta$ should satisfy $\beta\geq 1-\alpha$.
In fact, we have a rough observation from the coupling system of temperature $\theta$ and vorticity $\omega$, where $\omega$ is defined by
$\omega:=\mathrm{curl} u=\partial_{1}u^{2}-\partial_{2}u^{1}$. The coupling system writes
\begin{equation*}
 \begin{cases}\partial_{t}\omega+u\cdot\nabla \omega + \nu |D|^{\alpha}\omega= \partial_{1}\theta, \\
 \partial_{t}\theta+u\cdot\nabla \theta+ \kappa |D|^{\beta}\theta=0, \\
 \omega|_{t=0}=\omega^{0}:=\mathrm{curl}u^{0},\quad \theta|_{t=0}=\theta^{0}. \end{cases}
\end{equation*}
To get the key uniform estimates on $\omega$, the smoothing effect from the dissipation term  should at least roughly
compensate the loss of one derivative in $\theta$ in the vorticity equation with
the help of the diffusion effect in the temperature equation,
from which $\beta$ derivative in $\theta$ is gained.
Hence $\alpha+\beta\geq 1$ is needed. This is also the sense in which the case $\{\alpha=1,\nu>0, \kappa=0\}$ is called as a critical case.

Our goal in this paper is to understand the coupling effects between the dissipative terms which come both
in the velocity equation and the temperature equation, and their effects on the global existence, see Figure I.
For brevity, we always set $\nu=\kappa=1$ in the sequel.
We will modify the elegant argument introduced in \cite{ref HmidiKR-BNS} and \cite{ref HmidiKR-EB} to study  the
the global well-posedness of \eqref{eq 1.1} with $0<\beta<\alpha<1$.
More precisely, our main result is the following
\begin{theorem}\label{thm 1.1}
 Let $(\alpha,\beta)\in\Pi:=\big]\frac{6-\sqrt{6}}{4},1 \big[\times \big ] 1-\alpha, \min\{ \frac{7+2\sqrt{6}}{5}\alpha-2,\frac{\alpha(1-\alpha)}{\sqrt{6}-2\alpha},2-2\alpha\}\big]$
 (for $\Pi$ see Figure 1 in the sequel),
 $\theta^{0}\in H^{1-\alpha}\cap
 B^{1-\alpha}_{\infty,1}$ and $u^{0}$ be a divergence-free vector
 field belonging to $H^{1}\cap \dot W^{1,p}$ with $p\in ]\frac{2}{\beta+\alpha-1},\infty[$.
 Then the system \eqref{eq 1.1} has a unique global solution
 $(u,\theta)$ such that for every $\sigma\in [1,\frac{\alpha}{1-\alpha+2/p}[$,
 \begin{equation*}
  u\in L^{\infty}(\mathbb{R}^{+},H^{1}\cap \dot{W}^{1,p})\cap
  L^{\sigma}_{loc}(\mathbb{R}^{+},B^{1}_{\infty,1}) \quad and \quad \theta\in
  L^{\infty}(\mathbb{R}^{+},H^{1-\alpha}\cap
  B^{1-\alpha}_{\infty,1}).
 \end{equation*}
\end{theorem}

\begin{remark}\label{rem 1.1}
 Note that the case $\{\nu=\kappa=1,\alpha=1,\beta=0\}$ shares almost the same nature with the partially inviscid
 case $\{\nu=1,\kappa=0,\alpha=1\}$ studied in \cite{ref HmidiKR-BNS}, and the same method can be applied only by treating the damping term $\theta$ in the temperature equation
 as a harmless forcing term. Our results are motivated by this additional point and generalize the critical limiting case to the perturbed cases.
 This point is also the reason why the upper bound of $\beta$ occurs.
\end{remark}

The main idea in the proof of Theorem \ref{thm 1.1} is to use the internal structures of the system solved by $(\omega,\theta)$,
which is analogous to \cite{ref HmidiKR-BNS,ref HmidiKR-EB}. To get a first glance, we will neglect the nonlinear term here,
then the coupling system of $(\omega,\theta)$ reduces to
\begin{equation*}
 \partial_{t}\omega + |D|^{\alpha} \omega =\partial_{1} \theta, \quad \partial_{t}\theta + |D|^{\beta}\theta=0.
\end{equation*}
Thus
\begin{equation*}
 \partial_{t}\omega + |D|^{\alpha} (\omega-|D|^{-\alpha}\partial_{1}\theta)=0, \quad \partial_{t}\theta + |D|^{\beta}\theta=0.
\end{equation*}
Set $R_{\alpha}:=|D|^{-\alpha}\partial_{1}$, then
\begin{equation*}
 \partial_{t}(\omega-R_{\alpha}\theta) + |D|^{\alpha} (\omega-R_{\alpha}\theta)=|D|^{\beta-\alpha}\partial_{1}\theta, \quad \partial_{t}\theta + |D|^{\beta}\theta=0.
\end{equation*}
If roughly $\alpha\thicksim1$, $\beta\thicksim 0$, the forcing term $|D|^{\beta-\alpha}\partial_{1}\theta$ has much less loss of derivative than term $\partial_{1}\theta$ and indeed
we have some good estimates on $\omega-R_{\alpha}\theta$. These estimates will strongly help to obtain the needful estimates on $\omega$.

To prove Theorem \ref{thm 1.1}, we shall
study the new equation to get \textit{a priori} estimates on
$\omega-R_{\alpha}\theta$ and then return to obtain crucial
\textit{a priori} estimates on $\omega$. During this process, some
technical difficulty will be encountered. The first one is to
estimate the commutator $[R_{\alpha},u\cdot\nabla]$ which naturally
turns up when the nonlinear term is taken into account; another one
is to obtain estimates on $\omega$ from estimates on
$\omega-R_{\alpha}\theta$ (since in contrast with the Riesz
transform, $R_{\alpha}$ is not $L^{p}$-bounded and roughly contains
positive derivative of $1-\alpha$ power). We shall treat such
commutator estimates in Section 3 and yet we shall sufficiently use
the smooth effect (Proposition \ref{prop TDsf}) of the temperature
equation to overcome the another difficulty.

The paper is organized as follows. Section 2 is devoted to present some
preparatory results on Besov spaces. Some estimates about linear transport-diffusion equation are also given.
In Section 3, commutator estimates involving $R_{\alpha}$ are studied. Section 4 is the main part dedicated to the proof of Theorem \ref{thm 1.1}.
Finally, some technical lemmas are shown in Section 5.

\section{Priliminaries}
\setcounter{section}{2}\setcounter{equation}{0}

\subsection{Notations}
Throughout this paper the following notations will be used.
\\
$\bullet$ The notion $X\lesssim Y$ means that there exist a positive harmless constant $C$ such that $X\leq CY$. $X\thickapprox Y$ means that both
$X\lesssim Y$ and $Y\lesssim X$ are satisfied.
\\
$\bullet$ $\mathcal{S}$ denotes the Schwartz class, $\mathcal{S}'$ the
space of tempered distributions, $\mathcal{S}'/\mathcal{P}$ the quotient space of tempered distributions which modulo polynomials.
\\
$\bullet$ We use $\mathcal{F}f$ or $\widehat{f}$ to denote the Fourier transform of a tempered distribution $f$.
\\
$\bullet$ For any pair of operators $A$ and $B$ on some Banach space $\mathcal{X}$, the commutator $[A,B]$ is defined by $AB-BA$.
\\
$\bullet$ For every $k\in\mathbb{Z}^{+}$, the notion $\Phi_{k}$ denotes any function of the form
\begin{equation*}
 \Phi_{k}(t)= C_{0}\underbrace{\exp(\ldots\exp}_{k\; times}(C_{0}t)\ldots),
\end{equation*}
where $C_{0}$ depends on the related norms of the initial data and its value may be different from line to line up to some absolute constants.

\subsection{Littlewood-Paley decomposition and Besov spaces}

To define Besov space we need the following dyadic unity partition
(see e.g. \cite{ref chemin}). Choose two nonnegative radial
functions $\chi$, $\varphi\in C^{\infty}(\mathbb{R}^{n})$ be
supported respectively in the ball $\{\xi\in
\mathbb{R}^{n}:|\xi|\leq \frac{4}{3} \}$ and the shell $\{\xi\in
\mathbb{R}^{n}: \frac{3}{4}\leq |\xi|\leq
  \frac{8}{3} \}$ such that
\begin{equation*}
 \chi(\xi)+\sum_{j\geq 0}\varphi(2^{-q}\xi)=1, \quad
 \forall\xi\in\mathbb{R}^{n}; \qquad
 \sum_{q\in \mathbb{Z}}\varphi(2^{-q}\xi)=1, \quad \forall\xi\neq 0.
\end{equation*}
For all $f\in\mathcal{S}'(\mathbb{R}^{n})$ we define the
nonhomogeneous Littlewood-Paley operators
\begin{equation*}
 \Delta_{-1}f:= \chi(D)f; \;\forall q\in\mathbb{N} \quad
 \Delta_{q}f:= \varphi(2^{-q}D)f\;\; \mathrm{and} \; S_{q}f:=\sum_{-1\leq j\leq q-1} \Delta_{j}f.
\end{equation*}
The homogeneous Littlewood-Paley operators are defined as follows
\begin{equation*}
  \forall q\in\mathbb{Z},\quad \dot{\Delta}_{q}f:= \varphi(2^{-q}D)f,\quad  \dot S_{q}f:= \sum_{j\leq q-1}\dot\Delta_{j}f.
\end{equation*}
The paraproduct between two distributions $f$ and $g$ is defined by
\begin{equation*}
 T_{f}g:=\sum_{q\in\mathbb{N}}S_{q-1}f\Delta_{q}g.
\end{equation*}
Thus we have the following formal decomposition known as Bony's decomposition
\begin{equation*}
 fg=T_{f}g+T_{g}f+R(f,g),
\end{equation*}
where
\begin{equation*}
 R(f,g):=\sum_{q\geq -1}\Delta_{q}f\widetilde{\Delta}_{q}g, \quad \textrm{and}\quad \widetilde{\Delta}_{q}:=\Delta_{q-1}+\Delta_{q}+\Delta_{q+1}.
\end{equation*}

Now we introduce the definition of Besov spaces . Let $(p,r)\in
[1,\infty]^{2}$, $s\in\mathbb{R}$, the nonhomogeneous Besov space $B^{s}_{p,r}$ is defined as the set of tempered distribution $f$ such that
\begin{equation*}
  \norm{f}_{B^{s}_{p,r}}:=\norm{2^{qs}\norm{\Delta
  _{q}f}_{L^{p}}}_{\ell^{r}}<\infty,
\end{equation*}
The homogeneous space $ \dot{B}^{s}_{p,r}$ is the set of $f\in\mathcal{S}'(\mathbb{R}^{n})/\mathcal{P}(\mathbb{R}^{n})$ such that
\begin{equation*}
  \norm{f}_{\dot{B}^{s}_{p,r}}:=\norm{2^{qs}\norm{\dot{\Delta}
  _{q}f}_{L^{p}}}_{\ell^{r}(\mathbb{Z})}<\infty .
\end{equation*}
We point out that for all $s\in\mathbb{R}$, $B^{s}_{2,2}=H^{s}$ and
$\dot{B}^{s}_{2,2}=\dot{H}^{s}$.

Next we introduce two kinds of coupled space-time Besov spaces. The
first one $L^{\varrho}([0,T],B^{s}_{p,r})$, abbreviated by
$L^{\varrho}_{T}B^{s}_{p,r}$, is the set of tempered distribution $f$
such that
\begin{equation*}
  \norm{f}_{L^{\varrho}_{T}B^{s}_{p,r}}:=\norm{\norm{2^{qs}\norm{\Delta_{q}f}_{L^{p}}}_{\ell^{r}}}_{L^{\varrho}_{T}}<\infty.
\end{equation*}
The second one $\widetilde{L}^{\varrho}([0,T],B^{s}_{p,r})$,
abbreviated by $\widetilde{L}^{\varrho}_{T}B^{s}_{p,r}$, is the set of tempered
distribution $f$ satisfying
\begin{equation*}
  \norm{f}_{\widetilde{L}^{\varrho}_{T}B^{s}_{p,r}}:=\norm{2^{qs}\norm{\Delta_{q}f}_{L^{\varrho}_{T}L^{p}}
  }_{\ell^{r}}<\infty.
\end{equation*}
 Due to Minkowiski inequality, we immediately obtain
\begin{equation*}
 L^{\varrho}_{T}B^{s}_{p,r}\hookrightarrow
 \widetilde{L}^{\varrho}_{T}B^{s}_{p,r},\;  if\; r\geq \rho \quad and \quad
 \widetilde{L}^{\varrho}_{T}B^{s}_{p,r}\hookrightarrow
 L^{\varrho}_{T}B^{s}_{p,r},\;  if\; \varrho\geq r
\end{equation*}
We can similarly extend to the homogeneous ones
$L^{\varrho}_{T}\dot{B}^{s}_{p,r}$ and
$\widetilde{L}^{\varrho}_{T}\dot{B}^{s}_{p,r}$.

Berstein's inequality is fundamental in the analysis involving
Besov spaces (see e.g. \cite{ref chemin})
\begin{lemma}
Let $\mathcal{B}$ is a ball, $\mathcal{C}$ is a ring, $0\leq a\leq
b\leq \infty$. Then for $ k\in\mathbb{R}^{+}$, $\lambda>0$ there
exists a constant $C>0$ such that
\begin{equation*}
 \norm{|D|^{k}f}_{L^{b}}\leq C \lambda
 ^{k+n(\frac{1}{a}-\frac{1}{b})}\norm{f}_{L^{a}}  \quad if \; \textrm{supp}\; \mathcal{F}f
 \subset \lambda \mathcal{B},
\end{equation*}
\begin{equation*}
 C^{-1}\lambda
 ^{k}\norm{f}_{L^{a}}\leq\norm{|D|^{k}
 f}_{L^{a}}\leq C \lambda ^{k}\norm{f}_{L^{a}} \quad if \;
 \textrm{supp}\, \mathcal{F}f \subset \lambda \mathcal{C}.
\end{equation*}
\end{lemma}

\subsection{Transport-diffusion Equation }
In this subsection we shall collect some useful estimates for the smooth solutions of the following linear transport-diffusion equation
\begin{equation*}\label{eqTD}
  (TD)_{\beta}\quad \begin{cases} \partial_{t}\theta +
     u\cdot\nabla\theta+ |D|^{\beta}\theta=f,\quad \beta\in[0,1] \\
     \textrm{div} u=0, \quad  \theta|_{t=0}=\theta^{0}.
  \end{cases}
\end{equation*}

The $L^{p} $ estimate for (TD)$_{\beta}$ equation is
shown in \cite{ref AC-DC}
\begin{proposition}\label{prop MP}
Let $u$ be a smooth divergence-free vector field of $\mathbb{R}^{n}$ and $\theta$ be a smooth solution of (TD)$_{\beta}$. Then for every $p\in
[1,\infty]$ we have
\begin{equation}\label{eq TDMaxPrin}
 \norm{\theta(t)}_{L^{p}}\leq
 \norm{\theta_{0}}_{L^{p}}+\int^{t}_{0}\norm{f(\tau)}_{L^{p}}
 \,\textrm{d}\tau.
\end{equation}
\end{proposition}

The following smoothing effect is important in the proof.
\begin{proposition}\label{prop TDsf}
 Let $u$ be a smooth divergence-free vector field of $\mathbb{R}^{n}$ with vorticity $\omega$ and $\theta$ be a smooth solution of (TD)$_{\beta}$. Then for every
 ($p,\varrho$)$\in [2,\infty[\times [1,\infty]$ we have
 \begin{equation*}
  \mathrm{sup}_{q\in\mathbb{N}}2^{q\frac{\beta}{\varrho}} \norm{\Delta_{q}\theta}_{L^{\varrho}_{t}L^{p}}\lesssim_{\varrho,p} \norm{\theta^{0}}_{L^{p}}
  + \norm{\theta^{0}}_{L^{\infty}}\norm{\omega}_{L^{1}_{t}L^{p}} + \norm{f}_{L^{1}_{t}L^{p}}.
 \end{equation*}
\end{proposition}

\begin{remark}\label{rem TDsf}
 For $\varrho=1$, $\beta=1$ and $f=0$, the result has appeared in \cite{ref HmidiKR-EB}. Here with necessary modifications, this generalized case can be treated in a similar way.
\end{remark}

We also have the classical regularization effects as follows (see e.g. \cite{ref MiaoWu,ref HmidiKR-BNS,ref Hmidi-Ker}).
\begin{proposition}\label{propRF}
Let $-1<s<1$, $\varrho_{1}\leq \varrho$, $(p,r)\in [1,\infty]^{2}$ and $u$ be a divergence-free vector field belonging to
$L^{1}_{\textrm{loc}}(\mathbb{R}^{+};\textrm{Lip}(\mathbb{R}^{n}))$.
We consider a smooth solution $\theta$ of the equation
$(TD)_{\beta}$, then there exists $C>0$ such
that for each $t\in\mathbb{R}^{+}$,
\begin{equation*}\label{eqTDSE}
 \norm{\theta}_{\widetilde{L}^{\varrho}_{t}\dot B^{s+\frac{\beta}{\varrho}}_{p,r}}\leq
 C e^{C U(t) }\bigl(\norm{\theta_{0}}_{\dot B^{s}_{p,r}}+
 \norm{f}_{\widetilde{L}^{\varrho_{1}}_{t}\dot B^{s+\frac{\beta}{\varrho_{1}}-\beta}_{p,r}}\bigr),
\end{equation*}
where $U(t):=\int_{0}^{t}\norm{\nabla u(\tau)}_{L^{\infty}}\textrm{d}\tau$. Especially, if $\varrho=\infty$, we further have
\begin{equation*}
 \norm{\theta}_{\widetilde{L}^{\infty}_{t}B^{s}_{p,r}}\leq
 C e^{C U(t) }\bigl(\norm{\theta^{0}}_{B^{s}_{p,r}}+(1+t^{1-\frac{1}{\varrho_{1}}})
 \norm{f}_{\widetilde{L}^{\varrho_{1}}_{t}B^{s+\frac{\beta}{\varrho_{1}}-\beta}_{p,r}}\bigr).
\end{equation*}

\end{proposition}

\section{Modified Riesz transform and Commutators}
\setcounter{section}{3}\setcounter{equation}{0}

First we introduce a pseudo-differential operator $R_{\alpha}$
defined by $R_{\alpha}:=|D|^{-\alpha}\partial_{1}=|D|^{1-\alpha}R$,
$0<\alpha<1$, where $R:=\frac{\partial_{1}}{|D|}$ is the usual Riesz
transform. For convenience we call $R_{\alpha}$ as the modified Riesz
transform. We collect some useful properties of this operator as
follows.

\begin{proposition}\label{prop r-alpha-p}
 Let $0<\alpha<1$, $q\in\mathbb{N}$, $R_{\alpha}:=\frac{\partial_{1}}{|D|^{\alpha}}$ be the
 modified Riesz transform.
 \begin{enumerate}[(1)]
  \item
  Let $\chi\in \mathcal{D}(\mathbb{R}^{n})$. Then for every $(p,s)\in [1,\infty]\times]\alpha-1,\infty[$,
  $$\norm{|D|^{s}\chi(2^{-q}|D|)R_{\alpha}}_{\mathcal{L}(L^{p})}\lesssim 2^{q(s+1-\alpha)}.$$
  \item
  Let $\mathcal{C}$ be a ring. Then there exists $\phi\in
  \mathcal{S}(\mathbb{R}^{n})$ whose spectrum does not meet the origin such that
  $$R_{\alpha}f=2^{q(n+1-\alpha)} \phi(2^{q}\cdot)\star f$$
  for every $f$ with Fourier
  variable supported on $2^{q}\mathcal{C}$.
 \end{enumerate}
\end{proposition}

\begin{remark}\label{prop r-alpha-p}
 For the point (1), since $|D|^{s}\chi(2^{-q}|D|)R_{\alpha}=|D|^{s+1-\alpha}\chi(2^{-q}|D|)R$,
 it is naturally reduced to the case treated in Proposition 3.1 of \cite{ref HmidiKR-EB}.
 We here note that $|D|^{s}\chi(|D|)R_{\alpha}$ is a convolution operator with kernel $K$ satisfying
 $|K(x)|\lesssim 1/(1+|x|)^{n+s+1-\alpha }$ for all $x\in\mathbb{R}^{n}$.
 For the point (2), it can be achieved by a simple cut-off function technique.
\end{remark}

The following Lemma is useful in dealing with the commutator terms (see e.g. \cite{ref HmidiKR-EB}).
\begin{lemma}\label{lem commutator}
 Let $p\in [1,\infty]$, $m\geq p$, $\bar{m}=\frac{m}{m-1}$ be the dual number and $f,g,h$ belong to the suitable functional
 spaces. Then,
 \begin{equation}\label{eq commutator1}
  \norm{h\star(fg)-f(h\star g)}_{L^{p}}\leq \norm{xh}_{L^{\bar{m}}}\norm{\nabla
  f}_{L^{p}} \norm{g}_{L^{m}},
 \end{equation}
 \begin{equation}\label{eq commutator2}
  \norm{h\star(fg)-f(h\star g)}_{L^{p}}\leq \norm{xh}_{L^{1}}\norm{\nabla
  f}_{L^{\infty}} \norm{g}_{L^{p}}.
 \end{equation}

\end{lemma}

The next proposition consider the crucial commutators involving the
modified Riesz transform $R_{\alpha}$.

\begin{proposition}\label{prop cmt1}
 Let $\alpha\in]0,1[$, $u$ be a smooth divergence-free vector field of $\mathbb{R}^{n}$ and $\theta$ be
 a smooth scalar function. Then,
 \begin{enumerate}[(1)]
 \item for every $s\in ]0,\alpha[$ we have
  \begin{equation*}
   \norm{[R_{\alpha},u]\theta}_{H^{s}}\lesssim_{s,\alpha}\norm{\nabla
   u}_{L^{2}}\norm{\theta}_{B^{s-\alpha}_{\infty,2}}+\norm{u}_{L^{2}}\norm{\theta}_{L^{2}}.
  \end{equation*}
 In particular, if $u:=\Delta^{-1}\nabla^{\bot}\omega$ is given by
 the Biot-Savart law and $\omega:=\Gamma+R_{\alpha}\theta$, we have for every $s\in ]0,\alpha[$
 \begin{equation}\label{eq cmtEst1}
  \norm{[R_{\alpha},u]\theta}_{H^{s}}\lesssim_{s,\alpha}\norm{\Gamma}_{L^{2}}
  \norm{\theta}_{B^{s-\alpha}_{\infty,2}}+\norm{\theta}_{L^{\infty}}\norm{\theta}_{H^{s+1-2\alpha}}+\norm{u}_{L^{2}}\norm{\theta}_{L^{2}}.
 \end{equation}
 \item
 for every $(s,p,r)\in ]-1,\alpha[\,\times\,[2,\infty[\,\times\, [1,\infty]$ we have
 \begin{equation}\label{eq cmtEst2}
  \norm{[R_{\alpha},u\cdot\nabla]\theta}_{B^{s}_{p,r}} \lesssim_{s,\alpha} \norm{\nabla u}_{L^{p}}
  \big( \norm{\theta}_{B^{s+1-\alpha}_{\infty,r}} + \norm{\theta}_{L^{p}} \big).
 \end{equation}

 \end{enumerate}
\end{proposition}

\begin{proof}[Proof of Proposition \ref{prop cmt1}]
(1) We here only treat the special case to get \eqref{eq cmtEst1}. First due to Bony's decomposition we split
the commutator term into three parts
\begin{equation*}
\begin{split}
 [R_{\alpha},u]\theta & =\sum_{q\in
 \mathbb{N}}[R_{\alpha},S_{q-1}u]\Delta_{q}\theta+ \sum_{q\in
 \mathbb{N}}[R_{\alpha},\Delta_{q}u]S_{q-1}\theta + \sum_{q\geq -1}
 [R_{\alpha}, \Delta_{q}u]\widetilde{\Delta}_{q}\theta \\
 & := I + II + III.
\end{split}
\end{equation*}

$\bullet$ \textit{Estimation of $I$}.

Denote
$I_{q}:=[R_{\alpha},S_{q-1}u]\Delta_{q}\theta$. Since for each $q\in \mathbb{N}$ the Fourier transform of $I_{q}$ is supported in a
ring of size $2^{q}$, from the point (2)
of Proposition \ref{prop r-alpha-p} there exists $\phi\in \mathcal{S}$
whose spectrum is away from the origin such that
\begin{equation*}
 I_{q}=[\phi_{q}\star,S_{q-1}\Delta^{-1}\nabla^{\bot}\Gamma]\Delta_{q}\theta+
 [\phi_{q}\star,S_{q-1}\Delta^{-1}\nabla^{\bot}R_{\alpha}\theta]\Delta_{q}\theta,
\end{equation*}
where $\phi_{q}(x):=2^{q(n+1-\alpha)}\phi(2^{q}x)$. Taking advantage of
Lemma \ref{lem commutator} and the point (1) of Proposition \ref{prop
r-alpha-p} we obtain
\begin{equation*}
\begin{split}
 \norm{I_{q}}_{L^{2}}\leq & \norm{x \phi_{q}}_{L^{1}} \norm{\nabla
 S_{q-1}\Delta^{-1}\nabla^{\bot}\Gamma}_{L^{2}}\norm{\Delta_{q}\theta}_{L^{\infty}}+
 \\ & + \norm{x \phi_{q}}_{L^{1}} \norm{\nabla
 S_{q-1}\Delta^{-1}\nabla^{\bot}R_{\alpha}\theta}_{L^{\infty}}\norm{\Delta_{q}\theta}_{L^{2}}
 \\
 \lesssim & 2^{-q\alpha} \norm{\Gamma}_{L^{2}}\norm{\Delta_{q}\theta}_{L^{\infty}}
 + 2^{-q \alpha} 2^{q (1-\alpha)} \norm{\theta}_{L^{\infty}}\norm{\Delta_{q}\theta}_{L^{2}}.
\end{split}
\end{equation*}
Thus we directly have
\begin{equation*}
\begin{split}
 \norm{I}_{H^{s}} & \thickapprox
 \norm{2^{qs}\norm{I_{q}}_{L^{2}}}_{\ell^{2}} \\
  &\lesssim\norm{\Gamma}_{L^{2}}
  \norm{\theta}_{B^{s-\alpha}_{\infty,2}}+\norm{\theta}_{L^{\infty}}\norm{\theta}_{H^{s+1-2\alpha}}.
\end{split}
\end{equation*}

$\bullet$ \textit{Estimation of $II$}.

Denote
$II_{q}:=[R_{\alpha},\Delta_{q}u]S_{q-1}\theta$. As before we have
\begin{equation*}
 II_{q}=[\phi_{q}\star,\Delta_{q}\Delta^{-1}\nabla^{\bot}\Gamma]S_{q-1}\theta+
[\phi_{q}\star,\Delta_{q}\Delta^{-1}\nabla^{\bot}R_{\alpha}\theta]S_{q-1}\theta,
\end{equation*}
and again using Lemma \ref{lem commutator} we get
\begin{equation*}
\begin{split}
 \norm{II_{q}}_{L^{2}}\leq & \norm{x \phi_{q}}_{L^{1}} \norm{\nabla
 \Delta_{q}\Delta^{-1}\nabla^{\bot}\Gamma}_{L^{2}}\norm{S_{q-1}\theta}_{L^{\infty}}+
 \\ & + \norm{x \phi_{q}}_{L^{1}} \norm{\nabla
 \Delta_{q}\Delta^{-1}\nabla^{\bot}R_{\alpha}\theta}_{L^{2}}\norm{S_{q-1}\theta}_{L^{\infty}}
 \\
 \lesssim & 2^{-q\alpha} \norm{\Gamma}_{L^{2}}\norm{S_{q-1}\theta}_{L^{\infty}}
 + 2^{-q \alpha} 2^{q (1-\alpha)} \norm{\Delta_{q}\theta}_{L^{2}}\norm{\theta}_{L^{\infty}}.
\end{split}
\end{equation*}
Thus discrete Young inequality leads to for every $s<\alpha$
\begin{equation*}
\begin{split}
 \norm{II}_{H^{s}} & \thickapprox
 \norm{2^{qs}\norm{II_{q}}_{L^{2}}}_{\ell^{2}} \\
  &\lesssim\norm{\Gamma}_{L^{2}}
  \norm{\theta}_{B^{s-\alpha}_{\infty,2}}+\norm{\theta}_{L^{\infty}}\norm{\theta}_{H^{s+1-2\alpha}}.
\end{split}
\end{equation*}

$\bullet$ \textit{Estimation of $III$}.

We divide the term into three parts
\begin{equation*}
\begin{split}
 III &= \sum_{q\geq 0}R_{\alpha}(\Delta_{q}
 u\widetilde{\Delta}_{q}\theta) + \sum_{q\geq 0}\Delta_{q}
 u (R_{\alpha}\widetilde{\Delta}_{q}\theta) + [R_{\alpha},\Delta_{-1}
 u]\widetilde{\Delta}_{-1}\theta \\
 &:= III^{1}+III^{2}+III^{3}.
\end{split}
\end{equation*}
From direct computations we have
\begin{equation*}
\begin{split}
 2^{js}\norm{\Delta_{j}III^{1}}_{L^{2}} & \lesssim 2^{j(s+1-\alpha)}
 \sum_{q\geq j-4,q\geq 0} \norm{\Delta_{q}u}_{L^{2}}
 \norm{\widetilde{\Delta}_{q}\theta}_{L^{\infty}} \\
 & \lesssim 2^{j(s+1-\alpha)}\sum_{q\geq j-4}(2^{-q}\norm{\Delta_{q}\Gamma
 }_{L^{2}}+2^{-q\alpha}\norm{\Delta_{q}\theta
 }_{L^{2}})\norm{\widetilde{\Delta}_{q}\theta}_{L^{\infty}}
 \\ & \lesssim \sum_{q\geq j-4}2^{(j-q)(s+1-\alpha)}\Big(2^{q(s-\alpha)}\norm{\widetilde{\Delta}_{q}\theta}_{L^{\infty}}\norm{\Gamma
 }_{L^{2}}+2^{q(s+1-2\alpha)}\norm{\Delta_{q}\theta
 }_{L^{2}}\norm{\theta}_{L^{\infty}}\Big).
\end{split}
\end{equation*}
Thus discrete Young inequality (needing $s+1-\alpha>0$) yields
\begin{equation*}
 \norm{III^{1}}_{H^{s}}\lesssim_{s,\alpha}\norm{\Gamma}_{L^{2}}
  \norm{\theta}_{B^{s-\alpha}_{\infty,2}}+\norm{\theta}_{L^{\infty}}\norm{\theta}_{H^{s+1-2\alpha}}.
\end{equation*}
For $III^{2}$, by using the point (1) of Proposition \ref{prop r-alpha-p} we obtain
\begin{equation*}
\begin{split}
 2^{js}\norm{\Delta_{j}III^{2}}_{L^{2}} & \lesssim 2^{js}
 \sum_{q\geq j-4,q\geq 0} \norm{\Delta_{q}u}_{L^{2}}
 \norm{R_{\alpha}\widetilde{\Delta}_{q}\theta}_{L^{\infty}} \\
 & \lesssim 2^{js}\sum_{q\geq j-4}\big(2^{-q}\norm{\Delta_{q}\Gamma
 }_{L^{2}}+2^{-q\alpha}\norm{\Delta_{q}\theta
 }_{L^{2}}\big)2^{q(1-\alpha)}\norm{\widetilde{\Delta}_{q}\theta}_{L^{\infty}}
 \\ & \lesssim \sum_{q\geq j-4}2^{(j-q)s}\Big(2^{q(s-\alpha)}\norm{\widetilde{\Delta}_{q}\theta}_{L^{\infty}}\norm{\Gamma
 }_{L^{2}}+2^{q(s+1-2\alpha)}\norm{\Delta_{q}\theta
 }_{L^{2}}\norm{\theta}_{L^{\infty}}\Big).
\end{split}
\end{equation*}
Using convolution inequality (needing $s>0$) again we have
\begin{equation*}
 \norm{III^{2}}_{H^{s}}\lesssim_{s,\alpha}\norm{\Gamma}_{L^{2}}
  \norm{\theta}_{B^{s-\alpha}_{\infty,2}}+\norm{\theta}_{L^{\infty}}\norm{\theta}_{H^{s+1-2\alpha}}.
\end{equation*}
For $III^{3}$, since $\Delta_{j}III^{3}=0$ for every $j\geq 3$, then from Bernstein inequality and Calder\'on-Zygmund theorem we immediately have
\begin{equation*}
\begin{split}
 \norm{III^{3}}_{H^{s}} & \lesssim  \norm{[R_{\alpha},\Delta_{-1}u]\widetilde{\Delta}_{-1}\theta}_{L^{2}} \\
 & \lesssim
  \norm{\Delta_{-1}u}_{L^{2}}(\norm{\widetilde{\Delta}_{-1}\theta}_{L^{2}}+
 \norm{R_{\alpha}\widetilde{\Delta}_{-1}\theta}_{L^{2}}) \\
 & \lesssim \norm{u}_{L^{2}} \norm{\theta}_{L^{2}}.
\end{split}
\end{equation*}
This concludes the estimate \eqref{eq cmtEst1}.

(2)Once again using Bony's decomposition yields
\begin{equation*}
\begin{split}
 [R_{\alpha},u\cdot\nabla]\theta & =\sum_{q\in
 \mathbb{N}}[R_{\alpha},S_{q-1}u\cdot\nabla ]\Delta_{q}\theta + \sum_{q\in
 \mathbb{N}}[R_{\alpha},\Delta_{q}u\cdot\nabla]S_{q-1}\theta + \sum_{q\geq -1}
 [R_{\alpha}, \Delta_{q}u\cdot\nabla]\widetilde{\Delta}_{q}\theta \\
 & := \mathrm{I} + \mathrm{II} + \mathrm{III}.
\end{split}
\end{equation*}
For $\mathrm{I}$, since for every $q\in\mathbb{N}$ the Fourier transform of $S_{q-1}u\Delta_{q}\theta$ is supported in a ring of size $2^{q}$, then
from the point (2) of Proposition \ref{prop r-alpha-p} and Lemma \ref{lem commutator}, we have for every $j\geq -1$
\begin{equation*}
\begin{split}
 \norm{\Delta_{j}\mathrm{I}}_{L^{p}} & \lesssim \sum_{|q-j|\leq 4} \norm{[\phi_{q}\star,S_{q-1}u\cdot\nabla]\Delta_{q}\theta}_{L^{p}} \\
 & \lesssim \sum_{|q-j|\leq 4} 2^{-q\alpha}\norm{\nabla u}_{L^{p}} 2^{q}\norm{\Delta_{q}\theta}_{L^{\infty}}\\
 & \lesssim c_{j}2^{-js}\norm{\nabla u}_{L^{p}} \norm{\theta}_{B^{s+1-\alpha}_{\infty,r}},
\end{split}
\end{equation*}
where $\phi_{q}(x):=2^{q(n+1-\alpha)}\phi(2^{q}x)$ with $\phi\in\mathcal{S}$ and $(c_{j})_{j\geq -1}$ with $\norm{c_{j}}_{\ell^{r}}=1$. Thus we obtain
\begin{equation*}
 \norm{\mathrm{I}}_{B^{s}_{p,r}} \lesssim \norm{\nabla u}_{L^{p}} \norm{\theta}_{B^{s+1-\alpha}_{\infty,r}}.
\end{equation*}
For $\mathrm{II}$, as above we have for every $s<\alpha$ and $j\geq -1$
\begin{equation*}
\begin{split}
 \norm{\Delta_{j}\mathrm{II}}_{L^{p}} & \lesssim \sum_{|q-j|\leq 4,q\in \mathbb{N}} \norm{[\phi_{q}\star,\Delta_{q}u\cdot\nabla]S_{q-1}\theta}_{L^{p}} \\
 & \lesssim \sum_{|q-j|\leq 4}  2^{-q\alpha}\norm{\nabla u}_{L^{p}}\norm{\nabla S_{q-1}\theta}_{L^{\infty}} \\
 & \lesssim \norm{\nabla u}_{L^{p}} \sum_{|q-j|\leq 4} 2^{-qs} \sum_{q'\leq q-2} 2^{(q'-q)(\alpha-s)}2^{q'(s+1-\alpha)} \norm{\Delta_{q'} \theta}_{L^{\infty}}\\
 & \lesssim c_{j} 2^{-js}\norm{\nabla u}_{L^{p}}\norm{\theta}_{B^{s+1-\alpha}_{\infty,r}},
\end{split}
\end{equation*}
with $(c_{j})_{j\geq -1}$ satisfying $\norm{c_{j}}_{\ell^{r}}=1$. Thus
\begin{equation*}
 \norm{\mathrm{II}}_{B^{s}_{p,r}} \lesssim \norm{\nabla u}_{L^{p}} \norm{\theta}_{B^{s+1-\alpha}_{\infty,r}}.
\end{equation*}
For $\mathrm{III}$, we further write
\begin{equation*}
 \mathrm{III}= \sum_{q\geq 0 } \mathrm{div}[R_{\alpha},\Delta_{q}u]\widetilde{\Delta}_{q}\theta  +
 [\partial_{i} R_{\alpha},\Delta_{-1}u^{i}]\widetilde{\Delta}_{-1}\theta:= \mathrm{III}^{1} + \mathrm{III}^{2}.
\end{equation*}
For every $s>-1$, we treat the term $\mathrm{III}^{1}$ as follows
\begin{equation*}
\begin{split}
 \norm{\Delta_{j}\mathrm{III^{1}}}_{L^{p}} & \leq \sum_{q\geq j-4,q\geq 0}
 \norm{\Delta_{j}\mathrm{div} R_{\alpha}(\Delta_{q}u\widetilde{\Delta}_{q}\theta)}_{L^{p}} +\sum_{q\geq j-4,q\geq 0}
 \norm{\Delta_{j}\mathrm{div} (\Delta_{q}u R_{\alpha}\widetilde{\Delta}_{q}\theta)}_{L^{p}} \\
 & \lesssim \norm{\nabla u}_{L^{p}} 2^{-js} \sum_{q\geq j-4 }\big (2^{(j-q)(s+2-\alpha)}+2^{(j-q)(s+1 )}\big ) 2^{q(s+1-\alpha)}\norm{ \widetilde{\Delta}_{q}\theta }_{L^{\infty}} \\
 & \lesssim c_{j} 2^{-js}  \norm{\nabla u}_{L^{p}}\norm{\theta}_{B^{s+1-\alpha}_{\infty,r}},
\end{split}
\end{equation*}
with $(c_{j})_{j\geq -1}$ satisfying $\norm{c_{j}}_{\ell^{r}}=1$. Thus
\begin{equation*}
 \norm{\mathrm{III}^{1}}_{B^{s}_{p,r}} \lesssim  \norm{\nabla u}_{L^{p}}\norm{\theta}_{B^{s+1-\alpha}_{\infty,r}}.
\end{equation*}
For the second term, from the spectral property, there exist $\chi'\in \mathcal{D}(\mathbb{R}^{n})$ such that
\begin{equation*}
 \mathrm{III}^{2}  = [\partial_{i} R_{\alpha} \chi'(D),\Delta_{-1}u^{i}]\widetilde{\Delta}_{-1}\theta.
\end{equation*}
The Proposition \ref{prop r-alpha-p} shows that $\partial_{i}R_{\alpha}\chi'(D)$ is a convolution operator with kernel $h'$
satisfying
$$|h'(x)|\leq C (1+|x|)^{-n-2+\alpha}, \quad \forall x\in\mathbb{R}^{n}.$$
Thus from the fact that $\Delta_{j}\mathrm{III}^{2}=0$ for every
$j\geq 3$ and by applying Lemma \ref{lem commutator} with $m=p\geq 2$, we have
\begin{equation*}
\begin{split}
 \norm{\mathrm{III^{2}}}_{B^{s}_{p,r}} & \lesssim \norm{ [h'\star, \Delta_{-1}u ]\widetilde{\Delta}_{-1}\theta}_{L^{p}} \\
 & \lesssim  \norm{xh'}_{L^{\bar{p}}}\norm{\nabla \Delta_{-1}u}_{L^{p}} \norm{\widetilde{\Delta}_{-1}\theta}_{L^{p}} \\
 & \lesssim \norm{\nabla u}_{L^{p}}\norm{\theta}_{L^{p}}.
\end{split}
\end{equation*}
This ends the proof of estimate \eqref{eq cmtEst2}.


\end{proof}

\section{Proof of Theorem \ref{thm 1.1}}
\setcounter{section}{4}\setcounter{equation}{0}

\subsection{A priori estimates}

\begin{proposition}\label{prop apes1}
 Let ($u$,$\theta$) be a solution of the Boussinesq-Navier-Stokes
 system \eqref{eq 1.1} such that $(u^{0},\theta^{0})\in L^{2}\times L^{2}$. Then for every $t\in\mathbb{R}^{+}$
 \begin{equation*}
  \norm{\theta(t)}^{2}_{L^{2}}+2\norm{\theta}^{2}_{L_{t}^{2}\dot{H}^{\frac{\beta}{2}}}\leq \norm{\theta^{0}}_{L^{2}}^{2},
 \end{equation*}
 \begin{equation*}
  \norm{u(t)}^{2}_{L^{2}}+\norm{u}^{2}_{L_{t}^{2}\dot{H}^{\frac{\alpha}{2}}}\leq C_{0}(1+t^{2}),
 \end{equation*}
 Besides if $\theta^{0}\in L^{p}$ for some $p\in [1,\infty]$, we further have
 \begin{equation*}
   \norm{\theta(t)}_{L^{p}}\leq \norm{\theta^{0}}_{L^{p}}.
 \end{equation*}
\end{proposition}

\begin{proof}[Proof of Proposition \ref{prop apes1}]
 The $L^{p}$ estimate for $\theta$ is a direct consequence of Proposition \ref{prop MP}. For the $L^{2}$ estimate of $\theta$, by taking a $L^{2}$-inner product with $\theta$
 in the temperature equation we have
 \begin{equation*}
  \frac{1}{2}\frac{d}{dt} \norm{\theta(t)}_{L^{2}}^{2} + \norm{\theta(t)}_{\dot H^{\frac{\beta}{2}}}^{2} = 0.
 \end{equation*}
Thus integrating in time leads to the desired estimate. For the $L^{2}$ estimate of $u$, from the standard $L^{2}$ energy estimate, we get
 \begin{equation*}
  \frac{1}{2}\frac{d}{dt} \norm{u(t)}_{L^{2}}^{2} + \norm{u(t)}_{\dot H^{\frac{\alpha}{2}}}^{2} \leq \norm{u(t)}_{L^{2}} \norm{\theta(t)}_{L^{2}}.
 \end{equation*}
Thus we obtain
 \begin{equation*}
  \norm{u(t)}_{L^{2}} \leq \norm{u^{0}}_{L^{2}}+ \int_{0}^{t}\norm{\theta(\tau)}_{L^{2}}\mathrm{d}\tau \leq \norm{u^{0}}_{L^{2}}+ \norm{\theta^{0}}_{L^{2}} t.
 \end{equation*}
Putting this inequality in the previous one yields
 \begin{equation*}
  \frac{1}{2}\frac{d}{dt} \norm{u(t)}_{L^{2}}^{2} + \norm{u(t)}_{\dot H^{\frac{\alpha}{2}}}^{2} \leq
  \norm{\theta^{0}}_{L^{2}}\big(\norm{u^{0}}_{L^{2}}+ \norm{\theta^{0}}_{L^{2}} t \big).
 \end{equation*}
Integrating in time again leads to the desired result.

\end{proof}

\begin{proposition}\label{prop apes2}
 Let $\frac{6-\sqrt{6}}{4}<\alpha< 1$, $1-\alpha<\beta\leq \min\{ \frac{7+2\sqrt{6}}{5}\alpha-2,\frac{\alpha(1-\alpha)}{\sqrt{6}-2\alpha},2-2\alpha\}$, ($u$,$\theta$)
 be a solution of the Boussinesq-Navier-Stokes
 system \eqref{eq 1.1} such that $\theta^{0}\in  H^{1-\alpha}\cap B^{1-\alpha}_{\infty,1}$ and $u^{0}\in H^{1}\cap \dot
 W^{1,p}$ with $p\in ]\frac{2}{\beta+\alpha-1},\infty[$. Then for every $\sigma\in[1,\frac{\alpha}{1-\alpha+2/p}[$
 \begin{equation}\label{eq kest}
  \norm{u}_{L^{\sigma}_{t}B^{1}_{\infty,1}}\leq
  \Phi_{3}(t),
 \end{equation}
  \begin{equation*}
  \norm{\theta}_{L^{\infty}_{t}( H^{1-\alpha}\cap
  B^{1-\alpha}_{\infty,1})}+ \norm{\omega}_{L^{\infty}_{t}(L^{2}\cap
  L^{p})}\leq  \Phi_{3}(t).
 \end{equation*}
\end{proposition}

\begin{proof}[Proof of Proposition \ref{prop apes2}]
Denote $\Gamma:=\omega-R_{\alpha}\theta$. Considering the vorticity
equation
\begin{equation*}
 \partial_{t}\omega+u\cdot\nabla\omega+|D|^{\alpha}\omega=\partial_{1}\theta,
\end{equation*}
and the acting of $R_{\alpha}$ on the temperature equation
\begin{equation*}
 \partial_{t}R_{\alpha}\theta+u\cdot\nabla
 R_{\alpha}\theta+|D|^{\beta}R_{\alpha}\theta=-[R_{\alpha},u\cdot\nabla]\theta,
\end{equation*}
we directly have
\begin{equation}\label{eq eqInter}
 \partial_{t}\Gamma+u\cdot\nabla\Gamma+|D|^{\alpha}\Gamma=[R_{\alpha},u\cdot\nabla]\theta+|D|^{\beta}R_{\alpha}\theta.
\end{equation}
To obtain the key estimate \eqref{eq kest} with $\sigma=1$, the procedure below is
that we first obtain some "good" estimates on $\Gamma$ through
studying the interim equation \eqref{eq eqInter}, and then
combining with the estimates of $\theta$ we return to some
appropriate estimates on $\omega$ which lead to our target.

$\bullet$\textit{Step 1: Estimation of $\norm{\Gamma}_{L^{\infty}_{t}L^{2}}$}

From the classical energy method we get for every $s_{1}\in
[0,\frac{\alpha}{2}]$
\begin{equation*}
\begin{split}
 \frac{1}{2}\frac{d}{dt}\norm{\Gamma(t)}_{L^{2}}^{2}+\norm{\Gamma(t)}_{\dot
 H^{\frac{\alpha}{2}}}^{2} &
 =\int_{\mathrm{R}^{2}}\mathrm{div}([R_{\alpha},u]\theta)(t,x)\Gamma(t,x)\mathrm{d}x+\int_{\mathrm{R}^{2}}|D|^{\beta-\alpha}\partial_{1}\theta(t,x)
 \Gamma(t,x)\mathrm{d}x \\
 & \leq \norm{[R_{\alpha},u]\theta(t)}_{\dot
 H^{1-\frac{\alpha}{2}}}\norm{\Gamma(t)}_{\dot
 H^{\frac{\alpha}{2}}}+\norm{\theta(t)}_{\dot H^{1+\beta-\alpha-s_{1}}}\norm{\Gamma(t)}_{\dot
 H^{s_{1}}}.
\end{split}
\end{equation*}
 Interpolation inequality and
Young inequality yield
\begin{equation*}
\begin{split}
 \norm{\theta(t)}_{\dot H^{1+\beta-\alpha-s_{1}}}\norm{\Gamma(t)}_{\dot
 H^{s_{1}}} & \lesssim \norm{\theta(t)}_{\dot H^{1+\beta-\alpha-s_{1}}}\norm{\Gamma(t)}_{\dot
 H^{\frac{\alpha}{2}}}^{\frac{2s_{1}}{\alpha}} \norm{\Gamma(t)}_{
 L^{2}}^{1-\frac{2s_{1}}{\alpha}} \\
 & \leq C  \norm{\theta(t)}_{\dot H^{1+\beta-\alpha-s_{1}}}^{2}
 +C  \norm{\Gamma(t)}_{ L^{2}}^{2}+
 \frac{1}{4}\norm{\Gamma(t)}_{\dot H^{\frac{\alpha}{2}}}^{2}.
\end{split}
\end{equation*}
Inserting this inequality into the previous one and using Young inequality again we have
\begin{equation}\label{eq estimate1}
 \frac{d}{dt}\norm{\Gamma(t)}_{L^{2}}^{2}+\norm{\Gamma(t)}_{\dot
 H^{\frac{\alpha}{2}}}^{2}\leq 2\norm{[R_{\alpha},u]\theta(t)}_{
 \dot H^{1-\frac{\alpha}{2}}}^{2}+ 2 C \norm{\theta(t)}_{\dot
 H^{1+\beta-\alpha-s_{1}}}^{2} + 2C  \norm{\Gamma(t)}_{ L^{2}}^{2}.
\end{equation}
Applying Proposition \ref{prop cmt1} and Proposition \ref{prop
apes1} we have for every $\alpha\in ]\frac{2}{3},1[$
\begin{equation*}
\begin{split}
 \norm{[R_{\alpha},u]\theta(t)}_{
 \dot H^{1-\frac{\alpha}{2}}} & \leq \norm{[R_{\alpha},u]\theta(t)}_{
  H^{1-\frac{\alpha}{2}}} \\
  & \lesssim
 \norm{\Gamma(t)}_{L^{2}}\norm{\theta(t)}_{B^{1-\frac{3\alpha}{2}}_{\infty,2}}+
 \norm{\theta(t)}_{H^{2-\frac{5\alpha}{2}}}\norm{\theta(t)}_{L^{\infty}} +
 \norm{u(t)}_{L^{2}}\norm{\theta(t)}_{L^{2}} \\
 & \lesssim \norm{\Gamma(t)}_{L^{2}}\norm{\theta(t)}_{L^{\infty}}+
 \norm{\theta(t)}_{H^{2-\frac{5\alpha}{2}}}\norm{\theta^{0}}_{L^{\infty}} +
 (1+t) \\
 & \lesssim \norm{\Gamma(t)}_{L^{2}}+
 \norm{\theta(t)}_{H^{2-\frac{5\alpha}{2}}}+
 (1+t).
\end{split}
\end{equation*}
Putting the upper estimate in
\eqref{eq estimate1} leads to
\begin{equation*}
 \frac{d}{dt}\norm{\Gamma(t)}_{L^{2}}^{2}+\norm{\Gamma(t)}_{\dot
 H^{\frac{\alpha}{2}}}^{2} \lesssim \norm{\Gamma(t)}_{L^{2}}^{2}+
   \norm{\theta(t)}_{H^{2-\frac{5\alpha}{2}}}^{2}+  \norm{\theta(t)}_{\dot
 H^{1+\beta-\alpha-s_{1}}}^{2}+
  (1+t^{2}).
\end{equation*}
Using Gronwall inequality we obtain
\begin{equation*}
 \norm{\Gamma(t)}_{L^{2}}^{2}+ \int_{0}^{t}\norm{\Gamma(\tau)}_{\dot
 H^{\frac{\alpha}{2}}}^{2}\mathrm{d}\tau \leq  C_{1} e^{C_{1}t}(1+t^{2}+\norm{\theta}_{L^{2}_{t}H^{2-\frac{5\alpha}{2}}}^{2}+
 \norm{\theta}_{L^{2}_{t}\dot
 H^{1+\beta-\alpha-s_{1}}}^{2}).
\end{equation*}
If $\frac{3}{4}<\alpha\leq \frac{4}{5}$, we choose
$s_{1}=\frac{\alpha}{2}$, and for $1-\alpha<\beta\leq 3\alpha-2$,
then clearly
$$
0\leq 2-\frac{5\alpha}{2}\leq \frac{\beta}{2}, \quad 0\leq1+\beta-\frac{3\alpha}{2}\leq \frac{\beta}{2} , $$from
Proposition \ref{prop apes1} and interpolation inequality we easily
get
\begin{equation*}
 \norm{\theta}_{L^{2}_{t}H^{2-\frac{5\alpha}{2}}}^{2}+
 \norm{\theta}_{L^{2}_{t}\dot
 H^{1+\beta-\frac{3\alpha}{2}}}^{2}\lesssim 1+ t.
\end{equation*}
If $\frac{4}{5}<\alpha< 1$, we choose $s_{1}=2-2\alpha\in
]0,\frac{\alpha}{2}[$, and for $1-\alpha<\beta\leq 2-2\alpha$, then
$$0\leq\beta-1+\alpha\leq \frac{\beta}{2} ,$$
we also get
\begin{equation*}
 \norm{\theta}_{L^{2}_{t}H^{2-\frac{5\alpha}{2}}}^{2}+
 \norm{\theta}_{L^{2}_{t}\dot
 H^{\beta-1+\alpha}}^{2}\lesssim 1+ t.
\end{equation*}
Hence for every
$(\alpha,\beta)\in\Pi_{2}:=\big ]\frac{3}{4},1 \big[\,\times\,\big]1-\alpha,\min\{3\alpha-2,2-2\alpha\}\big]$
we have
\begin{equation}\label{eq key-est1}
 \norm{\Gamma(t)}_{L^{2}}^{2}+ \int_{0}^{t}\norm{\Gamma(\tau)}^{2}_{\dot
 H^{\frac{\alpha}{2}}}\mathrm{d}\tau \leq \Phi_{1}(t).
\end{equation}

$\bullet$\textit{Step 2: Estimation of $\norm{\Gamma}_{L^{\infty}_{t}L^{\widetilde{r}}}$ for every $\widetilde{r}\in [2,r]$ and for some $r\in [2,4[$}

Multiplying \eqref{eq eqInter} by $|\Gamma|^{r-2}\Gamma$ and
integrating in the spatial variable we obtain for every $s_{2},s_{3} \in
]0,\frac{\alpha}{2}]$ ($s_{3}\leq s_{2}$ and both will be chosen later)
\begin{equation*}
\begin{split}
 & \frac{1}{r}\frac{d}{dt}\norm{\Gamma(t)}_{L^{r}}^{r} +
 \int_{\mathbb{R}^{2}}|D|^{\alpha}\Gamma |\Gamma|^{r-2}\Gamma(t)
 \mathrm{d} x \\\leq &
 \int_{\mathbb{R}^{2}}\mathrm{div}[R_{\alpha},u]\theta
 |\Gamma|^{r-2}\Gamma(t) \mathrm{d} x +
 \int_{\mathbb{R}^{2}}|D|^{\beta-\alpha}\partial_{1}\theta |\Gamma|^{r-2}\Gamma(t)\mathrm{d}
 x \\
 \leq &  \norm{[R_{\alpha},u]\theta(t)}_{\dot H^{1-s_{2}}}\norm{|\Gamma|^{r-2}\Gamma(t)
 }_{\dot H^{s_{2}}} + \norm{\theta(t)}_{\dot H^{1+\beta-\alpha-s_{3}}}\norm{|\Gamma|^{r-2}\Gamma(t)
 }_{\dot H^{s_{3}}}.
\end{split}
\end{equation*}
Lemma 3.3 in \cite{Ju} and continuous embedding $\dot
H^{\frac{\alpha}{2}}\hookrightarrow L^{\frac{4}{2-\alpha}}$ lead to
\begin{equation*}
 \int_{\mathbb{R}^{2}}|D|^{\alpha}\Gamma |\Gamma|^{r-2}\Gamma
 \mathrm{d} x \gtrsim
 \norm{|\Gamma|^{\frac{r}{2}}}_{\dot H^{\frac{\alpha}{2}}}^{2}\gtrsim
 \norm{|\Gamma|^{\frac{r}{2}}}_{L^{\frac{4}{2-\alpha}}}^{2}=
 \norm{\Gamma}_{L^{\frac{2r}{2-\alpha}}}^{r}.
\end{equation*}
By using Lemma \ref{lem interpolation} in the Appendix we also find
\begin{equation*}
 \norm{|\Gamma|^{r-2}\Gamma
 }_{\dot H^{s_{i}}} \lesssim \norm{\Gamma}_{L^{\frac{2r}{2-\alpha}}}^{r-2} \norm{\Gamma}_{\dot
  H^{s_{i}+(1-\frac{2}{r})(2-\alpha)}}, \quad i=2,3.
\end{equation*}
Collecting the upper estimates we have
\begin{equation*}
\begin{split}
 & \frac{d}{dt}\norm{\Gamma(t)}_{L^{r}}^{r} +
 c\norm{\Gamma(t)}_{L^{\frac{2r}{2-\alpha}}}^{r}  \\ \lesssim &  \Big(\norm{[R_{\alpha},u]\theta(t)}_{\dot H^{1-s_{2}}}\norm{\Gamma}_{\dot
  H^{s_{2}+(1-\frac{2}{r})(2-\alpha)}}
  + \norm{\theta(t)}_{\dot H^{1+\beta-\alpha-s_{3}}}\norm{\Gamma}_{\dot
  H^{s_{3}+(1-\frac{2}{r})(2-\alpha)}}\Big )\norm{\Gamma}_{L^{\frac{2r}{2-\alpha}}}^{r-2}
\end{split}
\end{equation*}
Then we choose $s_{2}$ such that
$s_{2}+(1-\frac{2}{r})(2-\alpha)=\frac{\alpha}{2}$, which calls for
$s_{2}=\frac{\alpha}{2}-(1-\frac{2}{r})(2-\alpha)\in ]0,\frac{\alpha}{2}]$,
this is plausible if $\alpha\in ]\frac{4r-8}{3r-4},1 [$ for
$ r \in [2,4[$. Since $s_{3}\leq s_{2}$, by
interpolation we have
\begin{equation*}
 \norm{\Gamma(t)}_{\dot H^{s_{3}+(1-\frac{2}{r})(2-\alpha)}}\lesssim \norm{\Gamma(t)}_{\dot
 H^{\frac{\alpha}{2}}}^{\delta}\norm{\Gamma(t)}_{L^{2}}^{1-\delta}
 \leq \Phi_{1}(t) \norm{\Gamma(t)}_{\dot H^{\frac{\alpha}{2}}}^{\delta}.
\end{equation*}
where $\delta:=\frac{2}{\alpha}(s_{3}+(1-\frac{2}{r})(2-\alpha))$.
Also noticing that if $\alpha\in]\frac{6r-8}{5r-4},1[$, we have $1-s_{2}\in ]0, \alpha[$, then from the point (1)
of Proposition \ref{prop cmt1} and estimate \eqref{eq key-est1} we
further get
\begin{equation*}
\begin{split}
  \norm{[R_{\alpha},u]\theta(t)}_{\dot H^{1-s_{2}}} &  \leq \norm{[R_{\alpha},u]\theta(t)}_{H^{1-s_{2}}}\\
   & \lesssim \norm{\Gamma(t)}_{L^{2}}
  \norm{\theta(t)}_{B^{1-s_{2}-\alpha}_{\infty,2}}+\norm{\theta(t)}_{L^{\infty}}\norm{\theta(t)}_{H^{2-2\alpha-s_{2}}}+\norm{u(t)}_{L^{2}}\norm{\theta(t)}_{L^{2}} \\
  & \lesssim \norm{\Gamma(t)}_{L^{2}}
  \norm{\theta(t)}_{L^{\infty}}+\norm{\theta^{0}}_{L^{\infty}}\norm{\theta(t)}_{H^{2-2\alpha-s_{2}}}+(1+t)
  \\ & \lesssim \Phi_{1}(t) + \norm{\theta(t)}_{H^{2-2\alpha-s_{2}}}.
\end{split}
\end{equation*}
Therefore,
\begin{equation*}
\begin{split}
 & \frac{d}{dt}\norm{\Gamma(t)}_{L^{r}}^{r} +
 c\norm{\Gamma(t)}_{L^{\frac{2r}{2-\alpha}}}^{r}  \\ \lesssim &  \big (\Phi_{1}(t) + \norm{\theta(t)}_{H^{2-2\alpha-s_{2}}}
 \big )\norm{\Gamma(t)}_{L^{\frac{2r}{2-\alpha}}}^{r-2}
  \norm{\Gamma(t)}_{\dot
  H^{\frac{\alpha}{2}}}+ \Phi_{1}(t)\norm{\theta(t)}_{\dot
  H^{1+\beta-\alpha-s_{3}}} \norm{\Gamma(t)}_{L^{\frac{2r}{2-\alpha}}}^{r-2}
  \norm{\Gamma(t)}_{\dot H^{\frac{\alpha}{2}}}^{\delta}.
\end{split}
\end{equation*}
According to the following Young inequality
$$|A_{1}A_{2}A_{3}|\leq C'|A_{1}|^{\frac{2r}{4-r\widetilde{\delta}}}+C'' |A_{2}|^{\frac{2}{\widetilde{\delta}}}+ \frac{c}{4} |A_{3}|^{\frac{r}{r-2}}, \quad \widetilde{\delta}\in ]0,1], $$
we obtain
\begin{equation*}
\begin{split}
 & \frac{d}{dt}\norm{\Gamma(t)}_{L^{r}}^{r} +
 \norm{\Gamma(t)}_{L^{\frac{2r}{2-\alpha}}}^{r}  \\ \lesssim &  \Phi_{1}(t) +
 \norm{\theta(t)}_{H^{2-2\alpha-s_{2}}}^{\frac{2r}{4-r}}+
  \norm{\Gamma(t)}_{\dot H^{\frac{\alpha}{2}}}^{2}+
  \begin{cases} \Phi_{1}(t)\norm{\theta(t)}_{\dot
  H^{1+\beta-\alpha-s_{3}}}^{\frac{2r}{4-\delta r}}, \quad if \frac{2r}{4-\delta r}\geq 2 \\
  \norm{\theta(t)}_{\dot
  H^{1+\beta-\alpha-s_{3}}}^{2}, \quad otherwise. \end{cases}
\end{split}
\end{equation*}
Integrating in time yields
\begin{equation*}
\begin{split}
 & \norm{\Gamma(t)}_{L^{r}}^{r} +
 \int_{0}^{t}\norm{\Gamma(\tau)}_{L^{\frac{2r}{2-\alpha}}}^{r}\mathrm{d} \tau \\  \lesssim &  \Phi_{1}(t) +
 \norm{\theta}_{L^{\frac{2r}{4-r}}_{t}H^{2-2\alpha-s_{2}}}^{\frac{2r}{4-r}}+
 \begin{cases} \Phi_{1}(t)\norm{\theta}_{L^{\frac{2r}{4-\delta r}}_{t}\dot
  H^{1+\beta-\alpha-s_{3}}}^{\frac{2r}{4-\delta r}}, \quad if \frac{2r}{4-\delta r}\geq 2 \\
  \norm{\theta}_{L^{2}_{t}\dot
  H^{1+\beta-\alpha-s_{3}}}^{2}, \quad otherwise. \end{cases}
\end{split}
\end{equation*}
Note that we have used \eqref{eq key-est1} in the above deduction, thus it means $ (\alpha,\beta)\in \Pi_{2}$ at least.

Let $r\in [2,4[$. If $\alpha\in
]\frac{9r-12}{8r-8},\frac{8r-8}{7r-4}]$, we choose
$s_{3}:=s_{2}=\frac{3r-4}{2r}\alpha+\frac{4}{r}-2$, and for $\beta\in
]1-\alpha,\frac{5r-4}{3r-4}\alpha-2]$, we have $$0\leq 2-2\alpha
-s_{2}\leq \frac{4-r}{2r}\beta, \quad 0\leq 1+\beta-\alpha
-s_{2}\leq \frac{4-r}{2r}\beta ,$$ from Proposition \ref{prop apes1}
and interpolation inequality we find
\begin{equation*}
  \norm{\theta}_{L^{\frac{2r}{4- r}}_{t}\dot
  H^{1+\beta-\alpha-s_{2}}}+ \norm{\theta}_{L^{\frac{2r}{4-r}}_{t}H^{2-2\alpha-s_{2}}} \lesssim 1+t.
\end{equation*}
If $\alpha\in ]\frac{8r-8}{7r-4},1[$, we choose $s_{3}:=2-2\alpha<s_{2}$, then
$\delta = \frac{2}{\alpha}(2-2\alpha+\frac{r-2}{r}(2-\alpha))$ and for  $\beta\in
]1-\alpha,\min\{\frac{1-\alpha}{\frac{4}{\alpha}(1-\frac{1}{r})-2},2-2\alpha\}]$ we also get
\begin{equation*}
 0\leq \beta-1+\alpha \leq \frac{4-\delta r}{2r}\beta, \quad  0\leq \beta-1+\alpha \leq \frac{\beta}{2},
\end{equation*}
thus
\begin{equation*}
  \norm{\theta}_{L^{\frac{2r}{4- \delta r}}_{t}\dot
  H^{1+\beta-\alpha-s_{3}}}+\norm{\theta}_{L^{2}_{t}\dot
  H^{1+\beta-\alpha-s_{3}}} \lesssim 1+t.
\end{equation*}
Note that as $r\in[2,4[$ increases, the scope of $(\alpha,\beta)$ will monotonously shrink (e.g. see Figure 1).
Hence for some $r\in [2,4[$, $(\alpha,\beta)\in \Pi_{r}:= \big]\frac{9r-12}{8r-8},1\big[\, \times\, \big]1-\alpha,
\min\{\frac{5r-4}{3r-4}\alpha-2,\frac{1-\alpha}{\frac{4}{\alpha}(1-\frac{1}{r})-2},2-2\alpha \} \big]$, and for every $\widetilde{r}\in [2,r]$ we have
\begin{equation}\label{eq key-est2}
 \norm{\Gamma(t)}_{L^{\widetilde{r}}}^{\widetilde{r}} + \int_{0}^{t}\norm{\Gamma(\tau)}_{L^{\frac{2\widetilde{r}}{2-\alpha}}}^{\widetilde{r}}\mathrm{d}\tau \leq \Phi_{1}(t).
\end{equation}

\begin{figure}[htbp]
\begin{center}
\includegraphics[scale=0.8]{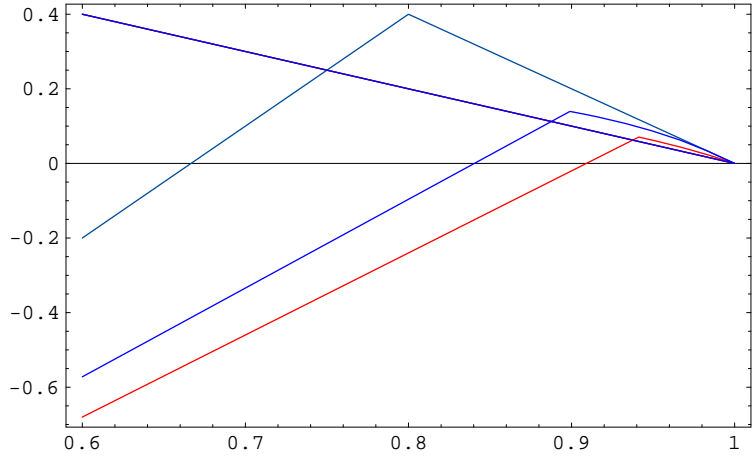}
\end{center}
\caption{$\Pi_{r}$ when $r=2,\frac{8+2\sqrt{6}}{5},3.$}
\vspace{0.1in}
\end{figure}

$\bullet$\textit{Step 3: Estimation of $\norm{\omega}_{L^{1}_{t}L^{\widetilde{r}}}$ for every $\widetilde{r}\in [2,r]$ and for some $r\in [2,4[$}

Since $\beta>1-\alpha$, there exists a fixed constant $\rho>1$ such that $\frac{\beta}{\rho}>1-\alpha$. From the explicit formula of $\Gamma$ we have for every  $\widetilde{r}\in [2,r]$
\begin{equation*}
\begin{split}
 \norm{\omega}_{L^{1}_{t}L^{\widetilde{r}}} & \leq \norm{\Gamma}_{L^{1}_{t}L^{\widetilde{r}}}+ \norm{R_{\alpha}\theta}_{L^{1}_{t}B^{0}_{\widetilde{r},1}} \\
 & \leq \Phi_{1}(t)+ t^{1-\frac{1}{\rho}}\norm{R_{\alpha}\theta}_{\widetilde{L}^{\rho}_{t}B^{0}_{\widetilde{r},1}} .
\end{split}
\end{equation*}
By a high-low frequency decomposition and a continuous embedding $B^{\frac{\beta}{\rho}}_{\widetilde{r},\infty}\hookrightarrow B^{1-\alpha}_{\widetilde{r},1}$ we find
\begin{equation*}
\begin{split}
 \norm{R_{\alpha}\theta}_{\widetilde{L}^{\rho}_{t}B^{0}_{\widetilde{r},1}} & \leq \norm{\Delta_{-1}R_{\alpha}\theta}_{\widetilde{L}^{\rho}_{t}B^{0}_{\widetilde{r},1}}
 + \norm{(Id-\Delta_{-1})\theta}_{\widetilde{L}^{\rho}_{t}B^{1-\alpha}_{\widetilde{r},1}} \\
 & \lesssim \norm{\Delta_{-1}\theta}_{L^{\rho}_{t}L^{\widetilde{r}}}+ \norm{(Id-\Delta_{-1})\theta}_{\widetilde{L}^{\rho}_{t}B^{\frac{\beta}{\rho}}_{\widetilde{r},\infty}} \\
 & \lesssim t^{\frac{1}{\rho}}\norm{\theta^{0}}_{L^{\widetilde{r}}} + \mathrm{sup}_{q\in\mathbb{N}}2^{q\frac{\beta}{\rho}} \norm{\Delta_{q}\theta}_{L^{\rho}_{t}L^{\widetilde{r}}}.
\end{split}
\end{equation*}
Inserting this estimate into the previous one and applying  Proposition \ref{prop TDsf} we obtain
\begin{equation*}
 \norm{\omega}_{L^{1}_{t}L^{\widetilde{r}}}\leq \Phi_{1}(t) + C t^{1-\frac{1}{\rho}} \norm{\omega}_{L^{1}_{t}L^{\widetilde{r}}},
\end{equation*}
where $C$ is an absolute constant depending only on $\widetilde{r},\rho$ and $\norm{\theta^{0}}_{L^{\infty}}$. If $C t^{1-\frac{1}{\rho}}=\frac{1}{2}$, equivalently,
$t=(\frac{1}{2C})^{\rho/(\rho-1)}:=T_{0}$, then for every $t\leq T_{0}$
$$ \norm{\omega}_{L^{1}_{t}L^{\widetilde{r}}}\leq \Phi_{1}(t). $$
Furthermore, if we evolve the system \eqref{eq 1.1} from the initial data $(u(T_{0}),\theta(T_{0}))$, then using the time translation invariance and the fact that
$\norm{\theta(T_{0})}_{L^{\widetilde{r}}}\leq \norm{\theta^{0}}_{L^{\widetilde{r}}}$, we have for every $t\leq T_{0}$
$$ \norm{\omega}_{L^{1}_{[T_{0},T_{0}+t]}L^{\widetilde{r}}}\leq \Phi_{1}(T_{0}+t). $$
Iterating like this, we finally get for every $t\in\mathbb{R}^{+}$
\begin{equation}\label{eq omegaL1Lr}
\norm{\omega}_{L^{1}_{t}L^{\widetilde{r}}}\leq \Phi_{1}(t).
\end{equation}

$\bullet$\textit{Step 4: Estimation of $\norm{\Gamma}_{L^{\sigma}_{t}B^{\frac{2}{r}}_{r,1}}$ for $\sigma \in[1,\frac{\alpha}{1-\alpha+2/r}[$ and $r=r_{0}:=\frac{8+2\sqrt{6}}{5}$}

Set $\Gamma_{q}:=\Delta_{q}\Gamma$ for every $q\in \mathbb{N}$. Considering the acting of frequency localization operator $\Delta_{q}$
on the equation \eqref{eq eqInter} we get
\begin{equation*}
\begin{split}
 \partial_{t}\Gamma_{q} + u\cdot\nabla\Gamma_{q}+|D|^{\alpha}\Gamma_{q} &=
 -[\Delta_{q},u\cdot\nabla]\Gamma+\Delta_{q}([R_{\alpha},u\cdot\nabla]\theta)+\Delta_{q}|D|^{\beta}R_{\alpha}\theta\\
 & := f_{q}.
\end{split}
\end{equation*}
Since $\Gamma_{q}$ is real-valued, then after multiplying the upper equation by $|\Gamma_{q}|^{r-2}\Gamma_{q}$ and integrating in the spatial variable we obtain
\begin{equation*}
 \frac{1}{r}\frac{d}{dt}\norm{\Gamma_{q}(t)}_{L^{r}}^{r} + \int_{\mathbb{R}^{2}}|D|^{\alpha}\Gamma_{q}|\Gamma_{q}|^{r-2}\Gamma \mathrm{d} x \leq \norm{\Gamma_{q}(t)}_{L^{r}}^{r-1}
 \norm{f_{q}(t)}_{L^{r}}.
\end{equation*}
Taking advantaging of the following generalized Bernstein inequality (see \cite{ref ChenMZ})
\begin{equation*}
 \int_{\mathbb{R}^{2}}|D|^{\alpha}\Gamma_{q}|\Gamma_{q}|^{r-2}\Gamma \mathrm{d} x \geq c 2^{q \alpha} \norm{\Gamma_{q}}_{L^{r}}^{r},
\end{equation*}
with some positive constant $c$ independent of $q$, we have
\begin{equation*}
 \frac{1}{r}\frac{d}{dt}\norm{\Gamma_{q}(t)}_{L^{r}}^{r} +  c 2^{q \alpha} \norm{\Gamma_{q}(t)}_{L^{r}}^{r} \leq \norm{\Gamma_{q}(t)}_{L^{r}}^{r-1}\norm{f_{q}(t)}_{L^{r}}.
\end{equation*}
Thus
\begin{equation*}
 \norm{\Gamma_{q}(t)}_{L^{r}}\leq e^{-ct2^{q \alpha}} \norm{\Gamma_{q}^{0}}_{L^{r}}+\int_{0}^{t} e^{-c(t-\tau)2^{q \alpha}}\norm{f_{q}(\tau)}_{L^{r}}\mathrm{d}\tau.
\end{equation*}
By taking the $L^{\sigma}([0,t])$ norm and by using the Young inequality we find for every $q\in \mathbb{N}$
\begin{equation}\label{eq main-4-key}
\begin{split}
 2^{q\frac{2}{r} }\norm{\Gamma_{q}}_{L^{\sigma}_{t}L^{r} }  \lesssim &  2^{q(\frac{2}{r}- \frac{\alpha}{\sigma} )}\norm{\Gamma_{q}^{0}}_{L^{r}}
 + 2^{q(\frac{2}{r}+1-\alpha- \frac{\alpha}{\sigma})} \int_{0}^{t}2^{q(\alpha-1)}\norm{[\Delta_{q},u\cdot\nabla]\Gamma(\tau)}_{L^{r}}\mathrm{d}\tau \\
 & + 2^{q(\frac{2}{r}+1-\alpha- \frac{\alpha}{\sigma})} \int_{0}^{t}2^{q(\alpha-1)}\norm{[R_{\alpha},u\cdot\nabla]\theta}_{L^{r}}\mathrm{d}\tau
 +  2^{q(\frac{2}{r}+1+\beta-\alpha- \frac{\alpha}{\sigma})}\norm{\Delta_{q}\theta}_{L^{1}_{t}L^{r}}.
\end{split}
\end{equation}
For the fourth term of the RHS, by using Proposition \ref{prop TDsf}, Proposition \ref{prop apes1} and estimate \eqref{eq omegaL1Lr} we get for each $q\in \mathbb{N}$
\begin{equation}\label{eq main-rhs4}
 \norm{\Delta_{q}\theta}_{L^{1}_{t}L^{r}} \lesssim 2^{-q\beta}(\norm{\theta^{0}}_{L^{r}}
  + \norm{\theta^{0}}_{L^{\infty}}\norm{\omega}_{L^{1}_{t}L^{r}}) \leq 2^{-q\beta}\Phi_{1}(t).
\end{equation}
For the third term of the RHS, we apply estimate \eqref{eq cmtEst2} with $s=\alpha-1$, Proposition \ref{prop apes1} and estimate \eqref{eq omegaL1Lr} to obtain
\begin{equation}\label{eq main-rhs3}
\begin{split}
 \int_{0}^{t}2^{q(\alpha-1)}\norm{[R_{\alpha},u\cdot\nabla]\theta(\tau)}_{L^{r}}\mathrm{d}\tau & \lesssim
  \int_{0}^{t}\norm{\nabla u(\tau)}_{L^{r}}\big(\norm{\theta(\tau)}_{L^{\infty}}+ \norm{ \theta (\tau)}_{L^{r}} \big)\mathrm{d}\tau \\
  & \lesssim \norm{\omega}_{L^{1}_{t}L^{r}}\norm{\theta^{0}}_{L^{\infty}\cap L^{r}} \\
  & \leq \Phi_{1}(t).
\end{split}
\end{equation}
For the second term of the RHS, in view of part (1) of Lemma \ref{lem commutator-EST} and the specific relationship between $u$ and $\theta$ we infer for every $q\in\mathbb{N}$
\begin{equation}\label{eq main-rhs2}
\begin{split}
 2^{q(\alpha-1)}\norm{[\Delta_{q},u\cdot\nabla]\Gamma(t)}_{L^{r}} & \lesssim (\norm{\nabla u(t)}_{B^{\alpha-1}_{r,\infty}} + \norm{u(t)}_{L^{2}})\norm{\Gamma(t)}_{B^{0}_{\infty,\infty}} \\
 & \lesssim (\norm{\Gamma(t)}_{L^{r}} + \norm{\theta(t)}_{L^{r}} + 1+t )\norm{\Gamma(t)}_{B^{\frac{2}{r}}_{r,1}} \\
 & \leq \Phi_{1}(t)\norm{\Gamma(t)}_{B^{\frac{2}{r}}_{r,1}}.
\end{split}
\end{equation}
To make the sum in the sequel summable, we need $\frac{2}{r}+1-\alpha-\frac{\alpha}{\sigma}<0$, that is,
$$1\leq \sigma <\frac{\alpha}{1-\alpha+2/r},\quad \max\Big \{\frac{2+r}{2r},\frac{9r-12}{8r-8} \Big\}<\alpha<1,\quad 2\leq r<4 .$$
Since for $r\in[2,4[$ the function $\frac{2+r}{2r}$ is monotonously decreasing and $\frac{9r-12}{8r-8}$ monotonously increasing, to obtain the largest scope of $\alpha$,
we have to choose $r=r_{0}:=\frac{8+2\sqrt{6}}{5}$ such that $\frac{2+r}{2r}=\frac{9r-12}{8r-8}=\frac{6-\sqrt{6}}{4}$. This leads to
\begin{equation*}
(\alpha,\beta)\in \Pi=\Big]\frac{6-\sqrt{6}}{4},1\Big[ \,\times\,\Big ]1-\alpha, \min\{ \frac{7+2\sqrt{6}}{5}\alpha-2,\frac{\alpha(1-\alpha)}{\sqrt{6}-2\alpha},2-2\alpha\}\Big].
\end{equation*}
Let $Q\in \mathbb{N}$ be a number chosen later. By gathering estimates \eqref{eq main-4-key}-\eqref{eq main-rhs2} together we find
\begin{equation*}
\begin{split}
 \norm{\Gamma}_{\widetilde{L}^{\sigma}_{t}B^{\frac{2}{r}}_{r,1}} & = \sum_{q< Q} 2^{q\frac{2}{r}}\norm{\Gamma}_{L^{\sigma}_{t}L^{r}}
 +  \sum_{q\geq Q} 2^{q\frac{2}{r}}\norm{\Gamma}_{L^{\sigma}_{t}L^{r}} \\
 & \lesssim 2^{Q\frac{2}{r}} \Phi_{1}(t) + 2^{-Q(\frac{\alpha}{\sigma}-\frac{2}{r})}\norm{\Gamma^{0}}_{L^{r}}
 +2^{-Q(\frac{\alpha}{\sigma}+\alpha-1-\frac{2}{r})}(\Phi_{1}(t)+\Phi_{1}(t)\norm{\Gamma}_{L^{1}_{t}B^{\frac{2}{r}}_{r,1}}) \\
 & \leq \Phi_{1}(t)2^{Q\frac{2}{r}}  + 2^{-Q(\frac{\alpha}{\sigma}+\alpha-1-\frac{2}{r})} \Phi_{1}(t) \norm{\Gamma}_{\widetilde{L}^{\sigma}_{t}B^{\frac{2}{r}}_{r,1}}.
\end{split}
\end{equation*}
We choose $Q$ such that
$$2^{-Q(\frac{\alpha}{\sigma}+\alpha-1-\frac{2}{r})} \Phi_{1}(t) \approx \frac{1}{2},$$
thus we obtain for every $t\in\mathbb{R}^{+}$
\begin{equation}\label{eq main-4-main}
 \norm{\Gamma}_{\widetilde{L}^{\sigma}_{t}B^{\frac{2}{r}}_{r,1}}\leq \Phi_{1}(t).
\end{equation}
By embedding this immediately leads to
\begin{equation}\label{eq main-4-main2}
 \norm{\Gamma}_{\widetilde{L}^{\sigma}_{t}B^{0}_{\infty,1}}\leq \Phi_{1}(t).
\end{equation}

$\bullet$\textit{Step 5: Estimation of $\norm{u}_{L^{1}_{t}B^{1}_{\infty,1}}$ }

By virtue of estimate \eqref{eq main-4-main2} with $\sigma=1$ and continuous embedding $B^{0}_{\infty,1}\cap L^{2}\hookrightarrow L^{\widetilde{p}} $ for all
$\widetilde{p}\in [2,p]$ we get
\begin{equation*}
\begin{split}
 \norm{\omega}_{L^{1}_{t}L^{\widetilde{p}}} & \leq \norm{\Gamma}_{L^{1}_{t}(B^{0}_{\infty,1}\cap L^{2})} + \norm{R_{\alpha} \theta}_{L^{1}_{t}B^{0}_{\widetilde{p},1} }\\
  & \leq \Phi_{1}(t) + t^{1-\frac{1}{\rho}} \norm{R_{\alpha}\theta }_{\widetilde{L}^{\rho}_{t}B^{0}_{\widetilde{p},1}}.
\end{split}
\end{equation*}
Thus in a similar way as obtaining \eqref{eq omegaL1Lr}, we have for every $\widetilde{p}\in]r_{0},p]$
\begin{equation*}
 \norm{\omega}_{L^{1}_{t}L^{\widetilde{p}}} \leq \Phi_{1}(t).
\end{equation*}
From Proposition \ref{prop TDsf}, we naturally deduce that
\begin{equation*}
   \sup_{q\in\mathbb{N}}2^{q\beta} \norm{\Delta_{q}\theta}_{L^{1}_{t}L^{\widetilde{p}}}\lesssim \norm{\theta^{0}}_{L^{\widetilde{p}}}
  + \norm{\theta^{0}}_{L^{\infty}}\norm{\omega}_{L^{1}_{t}L^{\widetilde{p}}} \leq \Phi_{1}(t).
\end{equation*}
Since $\beta>1-\alpha$ and $p>\frac{2}{\alpha+\beta-1}$, there exists a $p_{\alpha}\in]r_{0}, p]$ such that
$\beta>1-\alpha+\frac{2}{p_{\alpha}}$. Thus
\begin{equation*}
\begin{split}
 \norm{\omega}_{L^{1}_{t}B^{0}_{\infty,1}} & \leq \norm{\Gamma}_{L^{1}_{t}B^{0}_{\infty,1}}
 + \norm{R_{\alpha}\theta}_{L^{1}_{t}B^{0}_{\infty,1}} \\
 & \lesssim \Phi_{1}(t) + \norm{\Delta_{-1}R_{\alpha}\theta }_{L^{1}_{t}L^{\infty}}
 + \sum_{q\in \mathbb{N}} \norm{\Delta_{q}R_{\alpha}\theta}_{L^{1}_{t}L^{\infty}} \\
 & \lesssim \Phi_{1}(t) +  t \norm{\theta }_{L^{\infty}_{t}L^{2}}
 + \sum_{q\in \mathbb{N}} 2^{q(1-\alpha + \frac{2}{p_{\alpha}}-\beta)}
 \sup_{q\in \mathbb{N}}2^{q\beta} \norm{\Delta_{q}\theta}_{L^{1}_{t}L^{p_{\alpha}}} \\
 & \leq \Phi_{1}(t).
\end{split}
\end{equation*}
This immediately yields
\begin{equation}\label{eq-m-ap5-uL1B1oo1}
\begin{split}
 \norm{u}_{L^{1}_{t}B^{1}_{\infty,1}} & \lesssim \norm{\Delta_{-1}u}_{L^{1}_{t}L^{\infty}} + \sum_{q\in\mathbb{N}}\norm{\Delta_{q}\nabla u}_{L^{1}_{t}L^{\infty}}
 \lesssim \norm{u}_{L^{2}}+\norm{\omega}_{L^{1}_{t}B^{0}_{\infty,1}}\\
 & \leq \Phi_{1}(t).
\end{split}
\end{equation}

$\bullet$ \textit{Step 6: Estimation of $\norm{\theta}_{\widetilde{L}^{\infty}_{t}(H^{1-\alpha}\cap B^{1-\alpha}_{\infty,1})}$ and $\norm{\omega}_{L^{\infty}_{t}L^{p}}$}

By Proposition \ref{propRF} and estimate \eqref{eq-m-ap5-uL1B1oo1}, we directly obtain
\begin{equation}\label{eq-m-ap6-theta}
\begin{split}
 \norm{\theta}_{\widetilde{L}^{\infty}_{t}(H^{1-\alpha}\cap B^{1-\alpha}_{\infty,1})} & \lesssim e^{C \norm{\nabla u}_{L^{1}_{t}L^{\infty}}}
 \norm{\theta^{0}}_{H^{1-\alpha}\cap B^{1-\alpha}_{\infty,1}}\lesssim e^{C \norm{u}_{L^{1}_{t}B^{1}_{\infty,1}}} \\
 & \leq \Phi_{2}(t).
\end{split}
\end{equation}
For $p\in]\frac{2}{\alpha+\beta-1},\infty[$, in light of equation \eqref{eq eqInter} and Proposition \ref{prop MP} we find
\begin{equation*}
 \norm{\Gamma(t)}_{L^{p}}\leq \norm{\Gamma^{0}}_{L^{p}} + \int_{0}^{t}\norm{ [R_{\alpha}, u\cdot\nabla]\theta (\tau)}_{L^{p}}\mathrm{d}\tau
  + \int_{0}^{t}\norm{|D|^{\beta}R_{\alpha}\theta(\tau)}_{L^{p}}\mathrm{d} \tau.
\end{equation*}
For the first integral of the RHS, using estimate \eqref{eq cmtEst2} with $s=0$ yields
\begin{equation*}
\begin{split}
 \norm{[R_{\alpha}, u\cdot\nabla]\theta (\tau)}_{L^{p}} & \leq \norm{[R_{\alpha}, u\cdot\nabla]\theta (\tau)}_{B^{0}_{p,1}}\\
 & \lesssim\norm{\nabla u(\tau)}_{L^{p}}(\norm{\theta(\tau)}_{B^{1-\alpha}_{\infty,1}}
 + \norm{\theta(\tau)}_{L^{p}} ) \\
 & \leq \Phi_{2}(\tau) \norm{\omega(\tau)}_{L^{p}}.
\end{split}
\end{equation*}
For the second integral of the RHS, using Proposition \ref{propRF} again we infer
\begin{equation*}
\begin{split}
 \int_{0}^{t}\norm{|D|^{\beta}R_{\alpha}\theta(\tau)}_{L^{p}}\mathrm{d} \tau & \lesssim \norm{\theta}_{L^{1}_{t}L^{p}}+
 \norm{\theta}_{L^{1}_{t}\dot B^{1-\alpha+\beta}_{p,1}} \\
 & \lesssim \norm{\theta}_{L^{1}_{t}L^{p}}+  e^{C \norm{\nabla u}_{L^{1}_{t}L^{\infty}}} \norm{\theta^{0}}_{B^{1-\alpha}_{p,1}} \\
 & \leq \Phi_{2}(t).
\end{split}
\end{equation*}
Hence
\begin{equation*}
\begin{split}
 \norm{\omega(t)}_{ L^{p}} & \leq \norm{\Gamma(t)}_{ L^{p}} + \norm{R_{\alpha}\theta(t)}_{ L^{p}} \\
 & \leq \Phi_{2}(t) +\int_{0}^{t}\Phi_{2}(\tau)\norm{\omega(\tau)}_{L^{p}}\mathrm{d}\tau.
\end{split}
\end{equation*}
Gronwall inequality ensures
\begin{equation}\label{eq-m-ap6-omega}
 \norm{\omega(t)}_{ L^{p}}\leq \Phi_{3}(t).
\end{equation}
Taking estimates \eqref{eq-m-ap6-omega} and \eqref{eq-m-ap6-theta} into account, we return to \textit{Step 4} and further find for every $\sigma\in[1,\frac{\alpha}{1-\alpha+2/p}[$
\begin{equation}\label{eq-m-ap6-exten}
 \norm{ \Gamma}_{\widetilde{L}^{\sigma}_{t}B^{\frac{2}{p}}_{p,1}} + \norm{ u}_{\widetilde{L}^{\sigma}_{t}B^{1}_{\infty,1}} \leq \Phi_{3}(t).
\end{equation}
This finally ends the proof of Proposition \ref{prop apes2}.

\end{proof}

\subsection{Uniqueness}

We will prove a uniqueness result for the system \eqref{eq 1.1} with $(\alpha,\beta)\in \Pi$ in the following space
$$\mathcal{Y}_{T}:= L^{\infty}_{T}H^{1}\cap L^{1}_{T}B^{1}_{\infty,1}\times L^{\infty}_{T}L^{2}\cap L^{1}_{T}B^{1-\alpha}_{\infty,1}.$$
Let ($u_{i},\theta_{i}$)$\in\mathcal{Y}_{T}$ be two solutions of the system \ref{eq 1.1} with initial data ($u^{0}_{i},\theta^{0}_{i}$), $i=1,2$.
Set $\delta u:=u_{1}-u_{2}$, $\delta \theta:=\theta_{1}-\theta_{2}$ and $\delta p:=p_{1}-p_{2}$. Then
\begin{equation*}
\begin{split}
 \partial_{t}\delta u + u_{1}\cdot\nabla \delta u + |D|^{\alpha}\delta u + \nabla \delta p\; & =\, - \delta u\cdot\nabla u_{2} + \delta \theta e_{2} \\
 \partial_{t}\delta \theta + u_{1}\cdot\nabla \delta\theta + |D|^{\beta}\delta \theta\; & =\, - \delta u\cdot\nabla \theta_{2} \\
 (\delta u,\delta \theta)|_{t=0}\; &= \, (\delta u^{0},\delta \theta^{0}).
\end{split}
\end{equation*}
To estimate $\delta u$, by means of Lemma \ref{lem-appe-LVSE} (and its remark) in the appendix, we choose $\varrho=1$ for term $- \delta u\cdot\nabla u_{2}$ and
$\varrho=\infty$ for term $\delta \theta e_{2}$ to get for every $t\in[0,T]$
\begin{equation}\label{eq-m-uni-1}
\norm{\delta u(t)}_{B^{0}_{2,\infty}}\lesssim e^{C \norm{ u_{1}}_{L^{1}_{t}B^{1}_{\infty,1}} } \Big( \norm{\delta u^{0}}_{B^{0}_{2,\infty}}+
\int_{0}^{t}\norm{\delta u\cdot\nabla u_{2}(\tau)}_{B^{0}_{2,\infty}}\mathrm{d}\tau + (1+ t) \norm{\delta \theta}_{L^{\infty}_{t}B^{-\alpha}_{2,\infty}} \Big).
\end{equation}
For the integral term of the RHS, point (2) of Lemma \ref{lem commutator-EST} leads to
\begin{equation*}
 \norm{\delta u\cdot\nabla u_{2}}_{B^{0}_{2,\infty}}\lesssim \norm{u}_{L^{2}} \norm{u_{2}}_{B^{1}_{\infty,1}}.
\end{equation*}
Using the logarithmic interpolation inequality stated in Lemma 6.10 of \cite{ref HmidiKR-BNS} we have
\begin{equation*}
 \norm{\delta u}_{L^{2}} \lesssim \norm{\delta u}_{B^{0}_{2,\infty}} \log (e+\frac{1}{\norm{\delta u}_{B^{0}_{2,\infty}}}) \log(e+\norm{\delta u}_{H^{1}}).
\end{equation*}
Thus
\begin{equation}\label{eq-m-uni-2}
 \norm{\delta u\cdot\nabla u_{2}}_{B^{0}_{2,\infty}}\lesssim  \norm{u_{2}}_{B^{1}_{\infty,1}}\log(e+\norm{\delta u}_{H^{1}}) \mu(\norm{\delta u}_{B^{0}_{2,\infty}}),
\end{equation}
where $\mu(x):=x\log (e+\frac{1}{x})$. For the last term of the RHS of \eqref{eq-m-uni-1}, by virtue of Proposition \ref{propRF}, we have for every $t\in[0,T]$
\begin{equation}\label{eq-m-uni-3}
 \norm{\delta \theta}_{L^{\infty}_{t}B^{-\alpha}_{2,\infty}}\lesssim e^{C\norm{u_{1}}_{L^{1}_{t}B^{1}_{\infty,1}}}
 \Big(\norm{\delta \theta^{0}}_{B^{-\alpha}_{2,\infty}}+ \int_{0}^{t} \norm{\delta u\cdot\nabla \theta_{2}(\tau) }_{B^{-\alpha}_{2,\infty}}\mathrm{d}\tau \Big).
\end{equation}
Taking advantage of Lemma \ref{lem commutator-EST} and the logarithmic interpolation inequality again we obtain
\begin{equation}\label{eq-m-uni-4}
\begin{split}
 \norm{\delta u\cdot\nabla \theta_{2} }_{B^{-\alpha}_{2,\infty}} & \lesssim \norm{\delta u}_{L^{2}} \norm{\theta_{2}}_{B^{1-\alpha}_{\infty,1}} \\
 & \lesssim \norm{\theta_{2}}_{B^{1-\alpha}_{\infty,1}} \log(e+\norm{\delta u}_{H^{1}}) \mu(\norm{\delta u}_{B^{0}_{2,\infty}}).
\end{split}
\end{equation}
Denote $Y(t):=\norm{\delta u}_{L^{\infty}_{t}B^{0}_{2,\infty}}+ \norm{\delta \theta}_{L^{\infty}_{t}B^{-\alpha}_{2,\infty}} $. Gathering estimates \eqref{eq-m-uni-1}-\eqref{eq-m-uni-4} together yields
\begin{equation*}
 Y(t) \leq f(t)\Big(Y(0)+ \int_{0}^{t} \big( \norm{u_{2}(\tau)}_{B^{1}_{\infty,1}}+\norm{\theta_{2}(\tau)}_{B^{1-\alpha}_{\infty,1}} \big )\mu(Y(\tau))\mathrm{d}\tau \Big),
\end{equation*}
where $f(t)$ is an explicit function which continuously and increasingly depends on $\norm{(u_{i},\theta_{i})}_{\mathcal{Y}_{t}}$ and time $t$.
Since
\begin{equation*}
 \lim_{x\rightarrow 0+} \int_{x}^{1}\frac{1}{\mu(r)}\mathrm{d}r\geq\lim_{x\rightarrow 0+} \log(1+\log\frac{1}{x})= \infty,
\end{equation*}
then the classical Osgood lemma (see Theorem 5.2.1 in \cite{ref chemin}) ensures the uniqueness. Moreover, this lemma also shows some quantified estimates as follows
\begin{equation}\label{eq-m-uni-5}
 Y(0)\leq a(T)\Longrightarrow Y(T)\leq b(T)\big( Y(0) \big)^{\gamma(T)},
\end{equation}
where $a,b,\gamma$ are explicit functions depending continuously on $\norm{(u_{i},\theta_{i})}_{\mathcal{Y}_{T}}$ and time $T$.

\subsection{Existence}

First we smooth the data to get the following approximate system
\begin{equation}\label{eq gBn}
 \begin{cases}\partial_{t}u^{(n)}+u^{(n)}\cdot\nabla u^{(n)}+\nabla
 p^{(n)}+|D|^{\alpha}u^{(n)}=\theta^{(n)} e_{2}, \\
 \partial_{t}\theta^{(n)}+u^{(n)}\cdot\nabla \theta^{(n)}+|D|^{\beta}\theta^{(n)}=0, \\
 \textrm{div} u^{(n)}=0, \\
 u^{(n)}|_{t=0}=S_{n}u^{0},\quad \theta^{(n)}|_{t=0}=S_{n}\theta^{0}.
 \end{cases}
\end{equation}
Since $S_{n}u^{0},S_{n}\theta^{0}\in H^{s}$ for every $s\in\mathbb{R}$, from the classical theory of quasi-linear hyperbolic systems, we have the local well-posedness
of the approximate system. We also have a blowup criterion as follows: if the quantity
$\norm{\nabla u^{n}}_{L^{1}_{T}L^{\infty}}$ is finite, the time $T$ can be continued beyond. Then the \textit{a priori} estimate \eqref{eq kest} with $\sigma=1$ ensures that the solution $(u^{(n)},\theta^{(n)})$ is globally
defined. Moreover, we also have for $\sigma\in[1,\frac{\alpha}{1-\alpha+2/p}[$,
 \begin{equation*}
  \norm{u^{(n)}}_{L^{\sigma}_{T}B^{1}_{\infty,1}}+  \norm{u^{(n)}}_{L^{\infty}_{T}(H^{1}\cap
  \dot W^{1,p})} \leq
  \Phi_{3}(T),
 \end{equation*}
  \begin{equation*}
  \norm{\theta^{(n)}}_{L^{\infty}_{T}( H^{1-\alpha}\cap
  B^{1-\alpha}_{\infty,1})} \leq  \Phi_{2}(T).
 \end{equation*}
Thus there exists $(u,\theta)$ satisfying the above estimates such that $(u^{(n)},\theta^{(n)})$ weakly converges to $(u,\theta)$ up to the extraction of a subsequence.
Furthermore, from \eqref{eq-m-uni-5}, if
\begin{equation*}
 d_{n,m}:=\norm{(S_{n}-S_{m})u^{0}}_{B^{0}_{2,\infty}}+ \norm{(S_{n}-S_{m})\theta^{0}}_{B^{-\alpha}_{2,\infty}}\leq a(T)
\end{equation*}
then we have
\begin{equation*}
\norm{u^{(n)}-u^{(m)}}_{L^{\infty}_{T}B^{0}_{2,\infty}} + \norm{\theta^{(n)}-\theta^{(m)}}_{L^{\infty}_{T}B^{-\alpha}_{2,\infty}}\leq b(T)(d_{n,m})^{\gamma(T)}.
\end{equation*}
This means that $(u^{(n)})$ is of Cauchy and thus it converges strongly to $u$ in $L^{\infty}_{T}B^{0}_{2,\infty}$. By interpolation, we obtain the strong convergence of
$u^{(n)}$ to $u$ in $L^{2}([0,T]\times \mathbb{R}^{2})$. Thus $u^{(n)}\otimes u^{(n)}$ strongly converges in $L^{1}([0,T]\times \mathbb{R}^{2})$. But due to
that $\theta^{(n)}$ weakly converges to $\theta$ in $L^{2}([0,T]\times \mathbb{R}^{2})$, we have $u^{(n)}\theta^{(n)}$ converges weakly to $u\theta$. It then suffices to pass
to the limit in \eqref{eq gBn} and we finally get that $(u,\theta)$ is a solution of our original system \eqref{eq 1.1}.

\section{Appendix: Technical Lemmas}
\setcounter{section}{5}\setcounter{equation}{0}

\begin{lemma}\label{lem interpolation}
 Let $\gamma\in[2,\infty[$, $s\in ]0,1[$, $\alpha\in]\frac{\gamma-4}{\gamma-2},2[$. Then for
 every smooth function $f$ we have
 \begin{equation*}
  \norm{|f|^{\gamma-2}f}_{\dot H^{s}} \lesssim
  \norm{f}_{L^{\frac{2\gamma}{2-\alpha}}}^{\gamma-2} \norm{f}_{\dot
  H^{s+(\frac{n}{2}-\frac{n}{\gamma})(2-\alpha)}}.
 \end{equation*}
\end{lemma}

\begin{proof}[Proof of Lemma \ref{lem interpolation}]
 This result is a generalization of Lemma 6.9 in \cite{ref HmidiKR-BNS}, and here we sketch the proof. In fact, by Bernstein inequality,
it reduces to prove the following stronger result
\begin{equation*}
 \norm{|f|^{\gamma-2}f}_{\dot H^{s}} \lesssim \norm{f}_{L^{\frac{2\gamma}{2-\alpha}}}^{\gamma-2} \norm{f}_{\dot B^{s}_{\widetilde{\gamma},2}}.
\end{equation*}
where $\widetilde{\gamma}:=\frac{2\gamma}{\gamma-(\gamma-2)(2-\alpha)}$. For $s\in]0,1[$, we use the characterization of $\dot H^{s}$,
\begin{equation*}
 \norm{|f|^{\gamma-2}f}_{\dot H^{s}}^{2} \thickapprox \int_{\mathbb{R}^{n}}
 \frac{\norm{|f|^{\gamma-2}f(x+\cdot)-|f|^{\gamma-2}f(\cdot) }^{2}_{L^{2}}}{|x|^{n+2s}} \mathrm{d} x.
\end{equation*}
By using the simple inequality
\begin{equation*}
 \big||a|^{\gamma-2}a-|b|^{\gamma-2}b \big|\lesssim_{\gamma} |a-b| (|a|^{\gamma-2}+|b|^{\gamma-2}) , \quad \forall a,b\in\mathbb{R},
\end{equation*}
and H\"older inequality we have for every $\alpha\in] \frac{\gamma-4}{\gamma-2},2 [$
\begin{equation*}
\begin{split}
 \norm{|f|^{\gamma-2}f(x+\cdot) - |f|^{\gamma-2}f}_{L^{2}} & \lesssim \norm{f(x+\cdot)-f(\cdot)}_{L^{\widetilde{\gamma}}} \norm{|f|^{\gamma-2}}_{L^{\frac{2\gamma}{(\gamma-2)(2-\alpha)}}} \\
 & \lesssim \norm{f(x+\cdot)-f(\cdot)}_{L^{\widetilde{\gamma}}} \norm{f}_{L^{\frac{2\gamma}{2-\alpha}}}^{\gamma-2}.
\end{split}
\end{equation*}
From the characterization of homogeneous Besov space again we get the conclusion.

\end{proof}

Next we state some useful estimates in Besov framework.
\begin{lemma}\label{lem commutator-EST}
 Let $u$ be a smooth divergence-free vector field of $\mathbb{R}^{n}$ and $f$ be a smooth scalar function. Then
 \begin{enumerate}[(1)]
  \item
  for every $\alpha \in ]0,1[$ and $p\in [2,\infty]$
  \begin{equation*}
   \sup_{q\geq -1} 2^{q(\alpha-1)}\norm{[\Delta_{q},u\cdot\nabla]f}_{L^{p}} \lesssim_{\alpha}
   (\norm{\nabla u}_{B^{\alpha-1}_{p,\infty}}+ \norm{u}_{L^{2}}) \norm{f}_{B^{0}_{\infty,\infty}} .
  \end{equation*}
  \item
  for every $s\in [-1,0]$
  \begin{equation*}
   \norm{u\cdot\nabla f}_{B^{s}_{2,\infty}} \lesssim \norm{u}_{L^{2}} \norm{f}_{B^{s+1}_{\infty,1}}.
  \end{equation*}
 \end{enumerate}
\end{lemma}

\begin{proof}[Proof of Lemma \ref{lem commutator-EST}]
Note that point (2) is just the one in Lemma 6.10 of \cite{ref HmidiKR-BNS}, thus we only need to prove point (1). From Bony's decomposition we have
\begin{equation*}
\begin{split}
 & [\Delta_{q},u \cdot\nabla] f \\  =  & \sum_{|j-q| \leq 4} [\Delta_{q}, S_{j-1}u\cdot\nabla]\Delta_{j}f +
  \sum_{|j-q|\leq 4} [\Delta_{q}, \Delta_{j}u\cdot\nabla]S_{j-1}f + \sum_{j\geq q-4}  [\Delta_{q}\partial_{i},\Delta_{j}u^{i}]\widetilde{\Delta}_{j}f \\
  :=&  \mathrm{I}_{q}+ \mathrm{II}_{q}+ \mathrm{III}_{q}.
\end{split}
\end{equation*}
For $\mathrm{I}_{q}$, since $\Delta_{q}:=h_{q}(\cdot)\star = 2^{qn}h(2^{q}\cdot)\star$ with $h\in \mathcal{S}(\mathbb{R}^{n})$, then from \eqref{eq commutator1}
we get for every $\alpha<1$
\begin{equation*}
\begin{split}
 \norm{\mathrm{I}_{q}}_{L^{p}} & \lesssim \sum_{|j-q| \leq 4} \norm{x h_{q}}_{L^{1}} \norm{\nabla S_{j-1} u}_{L^{p}}2^{j}\norm{\Delta_{j}f}_{L^{\infty}} \\
 & \lesssim \norm{f}_{B^{0}_{\infty,\infty}} \norm{x h}_{L^{1}}  \sum_{|j-q| \leq 4} 2^{j-q} 2^{j(1-\alpha)}
 \sum_{k\leq j-2} 2^{(j-k)(\alpha-1)} 2^{k(\alpha-1)}\norm{\Delta_{k}\nabla u}_{L^{p}} \\
 & \lesssim 2^{q(1-\alpha)} \norm{\nabla u}_{B^{\alpha-1}_{p,\infty}} \norm{f}_{B^{0}_{\infty,\infty}},
\end{split}
\end{equation*}
thus
\begin{equation*}
 \sup_{q\geq -1} 2^{q(\alpha-1)}\norm{\mathrm{I}_{q}}_{L^{p}} \lesssim  \norm{\nabla u}_{B^{\alpha-1}_{p,\infty}} \norm{f}_{B^{0}_{\infty,\infty}}.
\end{equation*}
For $\mathrm{II}_{q}$, we also use \eqref{eq commutator1} to find
\begin{equation*}
\begin{split}
 \norm{\mathrm{II}_{q}}_{L^{p}} & \lesssim \sum_{|j-q| \leq 4} \norm{x h_{q}}_{L^{1}} \norm{\nabla \Delta_{j} u}_{L^{p}}\norm{S_{j-1} \nabla f}_{L^{\infty}} \\
 & \lesssim \norm{\nabla u}_{B^{\alpha-1}_{p,\infty}}  \sum_{|j-q| \leq 4} 2^{-q}  2^{j(1-\alpha)}
 \sum_{k\leq j-2}  2^{k}\norm{\Delta_{k} f}_{L^{\infty}} \\
 & \lesssim 2^{q(1-\alpha)} \norm{\nabla u}_{B^{\alpha-1}_{p,\infty}} \norm{f}_{B^{0}_{\infty,\infty}},
\end{split}
\end{equation*}
thus
\begin{equation*}
 \sup_{q\geq -1} 2^{q(\alpha-1)}\norm{\mathrm{II}_{q}}_{L^{p}} \lesssim  \norm{\nabla u}_{B^{\alpha-1}_{p,\infty}} \norm{f}_{B^{0}_{\infty,\infty}}.
\end{equation*}
For $\mathrm{III}_{q}$, we further write
\begin{equation*}
 \mathrm{III}_{q}=\sum_{j\geq q-4,j\in\mathbb{N}} [\Delta_{q}\partial_{i}, \Delta_{j} u^{i}] \widetilde{\Delta}_{j}f
 + [\Delta_{q}\partial_{i},\Delta_{-1}u^{i}]\widetilde{\Delta}_{-1}f := \mathrm{III}_{q}^{1}+\mathrm{III}_{q}^{2}.
\end{equation*}
For the first term, by direct computation we have for every $\alpha>0$
\begin{equation*}
\begin{split}
 \norm{\mathrm{III}_{q}^{1}}_{L^{p}} & \leq \sum_{j\geq q-4,j\in\mathbb{N}} \norm{\partial_{i}\Delta_{q}(\Delta_{j}u^{i}\widetilde{\Delta}_{j}f)}_{L^{p}}
 + \sum_{j\geq q-4,j\in\mathbb{N}} \norm{\Delta_{j}u^{i}\partial_{i} \Delta_{q}\widetilde{\Delta}_{j}f}_{L^{p}} \\
 & \lesssim  2^{q(1-\alpha)}\sum_{j\geq q-4,j\in\mathbb{N}} 2^{(q-j)\alpha}2^{j(\alpha-1)}\norm{\Delta_{j} \nabla  u}_{L^{p}} \norm{ \widetilde{\Delta}_{j}f}_{L^{\infty}} \\
 & \lesssim_{\alpha} 2^{q(1-\alpha)} \norm{\nabla u}_{B^{\alpha-1}_{p,\infty}} \norm{f}_{B^{0}_{\infty,\infty}},
\end{split}
\end{equation*}
thus
\begin{equation*}
 \sup_{q\geq -1} 2^{q(\alpha-1)}\norm{\mathrm{III}_{q}^{1}}_{L^{p}} \lesssim_{\alpha} \norm{\nabla u}_{B^{\alpha-1}_{p,\infty}}\norm{f}_{B^{0}_{\infty,\infty}}.
\end{equation*}
For the second term, due to $\mathrm{III}_{q}^{2}=0$ for every $q\geq 3$, we obtain for $p\geq 2$
\begin{equation*}
 \sup_{q\geq-1} 2^{q(\alpha-1)} \norm{\mathrm{III}_{q}^{2}}_{L^{p}} \lesssim \norm{ \Delta_{-1}u}_{L^{p}}\norm{\widetilde{\Delta}_{-1}f}_{L^{\infty}}
 \lesssim \norm{u}_{L^{2}}\norm{f}_{B^{0}_{\infty,\infty}}.
\end{equation*}
This concludes the proof.
\end{proof}

The following estimates on the linearized velocity equation is useful in the proof of the uniqueness part.
\begin{lemma}\label{lem-appe-LVSE}
Let $s\in]-1,1[$, $\varrho\in[1,\infty]$ and $v$ be a smooth divergence-free vector field of $\mathbb{R}^{n}$. If $u$ be a smooth solution of the linear system
\begin{equation*}
\partial_{t}u + v\cdot\nabla u+ |D|^{\alpha} u + \nabla p =f, \quad \mathrm{div} u=0.
\end{equation*}
with $\alpha\in[0,2]$ and $u|_{t=0}=u^{0}$, then for each $t\in\mathbb{R}^{+}$ we have
\begin{equation*}\label{eq-app- LVSE}
 \norm{u}_{L^{\infty}_{t}B^{s}_{2,\infty}}\leq
 C e^{C V(t) }\bigl(\norm{u^{0}}_{B^{s}_{2,\infty}}+(1+t^{1-\frac{1}{\varrho}})
 \norm{f}_{\widetilde{L}^{\varrho}_{t}B^{s+\frac{\alpha}{\varrho}-\alpha}_{2,\infty}}\bigr),
\end{equation*}
where $V(t):=\int_{0}^{t}\norm{\nabla v(\tau)}_{L^{\infty}}\textrm{d}\tau$.

\end{lemma}

\begin{remark}\label{rem-appe-LVSE}
 The proof can be done in a similar way as obtaining Proposition 4.3 in \cite{ref HmidiKR-BNS}. We also note that if $f=f_{1}+f_{2}$, one can choose different
 $\varrho_{1}, \varrho_{2}$ to suit $f_{1}, f_{2}$ respectively.
\end{remark}


\textbf{Acknowledgments:} The authors would like to thank Prof. T.Hmidi very much for his helpful discussion and suggestions.
The authors were partly supported by the NSF of China (No.10725102).


\end{document}